\documentclass[a4paper]{amsart}
\usepackage{amsmath}
\usepackage{amssymb}
\usepackage{amsthm,amsxtra}
\usepackage{cite}
\usepackage{enumerate}
\usepackage{enumitem}
\usepackage[mathscr]{eucal}
\usepackage{multicol}
\usepackage{adjustbox}
\usepackage{float}
\usepackage{tikz-cd}
\usepackage{tkz-graph}
\usepackage{tkz-berge}
\usetikzlibrary{positioning,arrows,shapes.geometric,trees}
\usepackage{graphicx}
\usepackage{mathtools}
\usetikzlibrary{positioning}
\usetikzlibrary{decorations.text}

\newtheorem{Def}{Definition}[section]
\newtheorem{Th}{Theorem}[section]
\newtheorem{Ex}{Example}[section]

\newtheorem{Lemma}{Lemma}[section]
\newtheorem{Prop}{Proposition}[section]
\newtheorem{Cor}{Corollary}[section]
\newtheorem{Rem}{Remark}[section]

{}
\newtheorem{Prob}{Problem}[section]

\DeclareMathOperator{\Int}{Int}

\DeclareMathOperator{\maxfin}{\sf maxfin}

\DeclareMathOperator{\PR}{\sf PR}
\DeclareMathOperator{\aL}{\sf aL}
\DeclareMathOperator{\aM}{\sf aM}
\DeclareMathOperator{\aS}{\sf aS}
\DeclareMathOperator{\aH}{\sf aH}
\DeclareMathOperator{\wL}{\sf wL}
\DeclareMathOperator{\wM}{\sf wM}
\DeclareMathOperator{\wS}{\sf wS}
\DeclareMathOperator{\wH}{\sf wH}
\DeclareMathOperator{\wSk}{\sf wS_k}
\DeclareMathOperator{\wHk}{\sf wH_k}
\DeclareMathOperator{\Uf}{U_{fin}}
\DeclareMathOperator{\S1}{S_1}
\DeclareMathOperator{\Sf}{S_{fin}}
\DeclareMathOperator{\n}{\sf N}
\begin{document}
\title[On certain weaker forms of the Scheepers property]{On certain weaker forms of the Scheepers property}

\author[ D. Chandra, N. Alam ]{ Debraj Chandra$^*$, Nur Alam$^*$ }
\newcommand{\acr}{\newline\indent}
\address{\llap{*\,}Department of Mathematics, University of Gour Banga, Malda-732103, West Bengal, India}
\email{debrajchandra1986@gmail.com, nurrejwana@gmail.com}

\thanks{ The second author
is thankful to University Grants Commission (UGC), New Delhi-110002, India for granting UGC-NET Junior Research Fellowship (1173/(CSIR-UGC NET JUNE 2017)) during the tenure of which this work was done.}


\maketitle

\begin{abstract}
We introduce the weaker forms of the Scheepers property, namely almost Scheepers ($\aS$), weakly Scheepers in the sense of Sakai ($\wS$) and weakly Scheepers in the sense of Ko\v{c}inac ($\wSk$).
We explore many topological properties of the weaker forms of the Scheepers property and present few illustrative examples to make distinction between these spaces. Certain situations are considered when all the weaker forms are equivalent. We also make investigations on the weak variations as considered in this paper concerning cardinalities.
In particular we observe that

\begin{enumerate}
  \item If every finite power of a space $X$ is $\aM$ (respectively, $\wM$), then $X$ is $\aS$ (respectively, $\wS$).
  \item Every almost Lindel\"{o}f space of cardinality less than $\mathfrak{d}$ is $\aS$.
  \item Let $X$ be Lindel\"{o}f and $\kappa<\mathfrak d$. If $X$ is a union of $\kappa$ many $\aH$ (respectively, $\wH$, $\wHk$) spaces, then $X$ is $\aS$ (respectively, $\wS$, $\wSk$).
  \item The Alexandroff duplicate $AD(X)$ of a space $X$ has the Scheepers property if and only if $AD(X)$ has the $\wSk$ property.
  \item If $AD(X)$ is $\aS$ (respectively, $\wS$), then $X$ is also $\aS$ (respectively, $\wS$).
\end{enumerate}

Besides, few observations on productively $\aS$, productively $\wS$ and productively $\wSk$ spaces are presented.
Some open problems are also given.

\end{abstract}

\noindent{\bf\keywordsname{}:} {Scheepers, almost Scheepers, weakly Scheepers, almost Alster, weakly Alster, productively almost Scheepers, productively weakly Scheepers.}

\section{Introduction}
The study of selection principles in set-theoretic topology is well known. A lot of research has been carried out investigating weaker forms of the Menger \cite{coc1} and Hurewicz \cite{coc1} properties, namely almost Menger (Hurewicz) and weakly Menger (Hurewicz) properties, for detailed information see \cite{PRTS,AHS,WMP,SMR-II,MTCP,PRS,WHP,survey1} and references therein. In this article we introduce certain weaker forms of the Scheepers property \cite{coc1,coc2}, namely almost Scheepers ($\aS$), weakly Scheepers in the sense of Sakai ($\wS$) and weakly Scheepers in the sense of Ko\v{c}inac ($\wSk$). The purpose of this study is to observe the relationship between the Scheepers property and its weaker variations, and also investigate the behaviour of these new relatives. The interrelationships between the notions considered here are outlined into implication diagrams (Figure~\ref{dig1} and Figure~\ref{dig2}).

The following is a summarization of what is done in this paper.
\begin{enumerate}[wide=0pt,label={\upshape(\arabic*)},leftmargin=*]
  \item Every paracompact (and also regular) $\aS$ space is Scheepers. But a hypocompact (and hence a paracompact) as well as a regular $\wSk$ ($\wS$) space need not be Scheepers.
  \item If every finite power of a space $X$ is $\aM$ (respectively, $\wM$), then $X$ is $\aS$ (respectively, $\wS$).
  \item Every almost Alster (respectively, weakly Alster) space is $\aS$ (respectively, $\wS$).
  \item Every almost Lindel\"{o}f space with cardinality less than $\mathfrak{d}$ is $\aS$.
  \item Let $X$ be Lindel\"{o}f and $\kappa<\mathfrak d$. If $X$ is a union of $\kappa$ many $\aH$ (respectively, $\wH$, $\wHk$) spaces, then $X$ is $\aS$ (respectively, $\wS$, $\wSk$).
  \item The Alexandroff duplicate $AD(X)$ of a space $X$ has the Scheepers property if and only if $AD(X)$ has the $\wSk$ property.
  \item If $AD(X)$ is $\aS$ (respectively, $\wS$), then $X$ is also $\aS$ (respectively, $\wS$).
  \item Every $H$-closed space is productively $\aS$ (productively $\wS$, productively $\wSk$).
  \item If a space $X$ satisfies $\S1(\mathcal{G}_K,\mathcal{G}_{\overline{\Gamma}})$ (respectively, $\S1(\mathcal{G}_K,\mathcal{G}_{D_\Gamma})$, $\S1(\mathcal{G}_K,\mathcal{G}_{\Gamma_D})$), then $X$ is productively $\aS$ (respectively, productively $\wS$, productively $\wSk$).
\end{enumerate}

The paper is organized as follows. In Section 3, it is observed that the Scheepers property is distinguishable from its weaker forms. Contrasting characterizations of these weaker variations in terms of regular-open sets as well as neighbourhood assignments are obtained. Certain situations are considered when all the weaker forms are equivalent. In \cite[Proposition 2.3]{MTCP}, the authors answered the question \cite[Question 13]{WMP} that under what conditions the $\aM$ and $\wM$ properties are equivalent. We improve the result \cite[Proposition 2.3]{MTCP} by showing that in the context of a $P$-space almost all the weak variations are identical (Theorem~\ref{T44} and Corollary~\ref{C29}). Later in this section, we present some observations on the weak variations (as considered in this paper) concerning critical cardinalities. We improve the result \cite[Theorem 2.5]{QM} by showing that if a Lindel\"{o}f space $X$ is written as a union of $\kappa$ ($<\mathfrak{d}$) many $H$-closed subspaces, then $X$ is $\aS$ (Corollary~\ref{C16}). In Section 4, preservation like properties of the new notions are investigated. We consider and study the productively properties, viz. productively $\aS$, productively $\wS$ and productively $\wSk$ spaces. In the later part of this section, we also study the behaviour of these notions under certain mappings (mainly weaker versions of continuous mapping). In Section 5, we explore few possibilities for further lines of investigation. In the final section, some open problems are posted.

\section{Preliminaries}
Throughout the paper $(X,\tau)$ stands for a topological space. For undefined notions and terminologies, see \cite{Engelking}.

Let $\mathcal{A}$ and $\mathcal{B}$ be families of subsets of a space $X$. Following \cite{coc1,coc2}, we define

\noindent $\S1(\mathcal{A},\mathcal{B})$: For each sequence $(\mathcal{U}_n)$ of elements of $\mathcal{A}$ there exists a sequence $(V_n)$ such that for each $n$ $V_n\in\mathcal{U}_n$ and $\{V_n : n\in\mathbb{N}\}\in\mathcal{B}$.\\

\noindent $\Sf(\mathcal{A},\mathcal{B})$: For each sequence $(\mathcal{U}_n)$ of elements of $\mathcal{A}$ there exists a sequence $(\mathcal{V}_n)$ such that for each $n$ $\mathcal{V}_n$ is a finite subset of $\mathcal{U}_n$ and $\cup_{n\in\mathbb{N}}\mathcal{V}_n\in\mathcal{B}$.\\

\noindent $\Uf(\mathcal{A},\mathcal{B})$: For each sequence $(\mathcal{U}_n)$ of elements of $\mathcal{A}$ there exists a sequence $(\mathcal{V}_n)$ such that for each $n$ $\mathcal{V}_n$ is a finite subset of $\mathcal{U}_n$ and $\{\cup\mathcal{V}_n : n\in\mathbb{N}\}\in\mathcal{B}$ or there is some $n$ such that $\cup\mathcal{V}_n=X$.

Let $\mathcal O$ denote the collection of all open covers of $X$. An open cover $\mathcal{U}$ of $X$ is said to be a $\gamma$-cover \cite{coc1,coc2} if $\mathcal{U}$ is infinite and for each $x\in X$, the set $\{U\in\mathcal{U} : x\notin U\}$ is finite. An open cover $\mathcal{U}$ of $X$ is said to be an $\omega$-cover \cite{coc1,coc2} if $X$ not in $\mathcal{U}$ and for each finite subset $F$ of $X$ there is a set $U\in\mathcal{U}$ such that $F\subseteq U$. An open cover $\mathcal{U}$ of $X$ is said to be a large cover \cite{coc1,coc2} if for each $x\in X$, the set $\{U\in\mathcal{U} : x\in U\}$ is infinite.
We use the symbol $\Gamma$, $\Omega$ and $\Lambda$ to denote the collection of all $\gamma$-covers, $\omega$-covers and large covers of $X$ respectively. Note that $\Gamma\subseteq\Omega\subseteq\Lambda\subseteq\mathcal{O}$.
An open cover $\mathcal{U}$ of $X$ is said to be weakly groupable \cite{cocVIII} if $X$ can be expressed as a countable union of finite, pairwise disjoint subfamilies $\mathcal{U}_n$, $n\in\mathbb{N}$, such that for each finite set $F\subseteq X$ we have $F\subseteq\cup\mathcal{U}_n$ for some $n$. The collection of all weakly groupable covers of $X$ is denoted by $\mathcal{O}^{wgp}$. A space $X$ is said to have the Menger (respectively, Scheepers, Hurewicz) property if $X$ satisfies the selection principle $\Sf(\mathcal{O},\mathcal{O})$ (respectively, $\Uf(\mathcal{O},\Omega)$, $\Uf(\mathcal{O},\Gamma)$) \cite{coc1,coc2}.

We now consider the following classes of covers (for similar types of covers, see \cite{NSM}).
\begin{enumerate}[leftmargin=*]
\item[$\overline{\mathcal{O}}$:] The family of all sets $\mathcal{U}$ of open subsets of $X$ such that $\{\overline{U} : U\in\mathcal{U}\}$ covers $X$.

\item[$\overline{\Omega}$:] The family of all sets $\mathcal{U}\in\overline{\mathcal{O}}$ such that each finite set $F\subseteq X$ is contained in $\overline{U}$ for some $U\in\mathcal{U}$.

\item[$\overline{\Gamma}$:] The family of all sets $\mathcal{U}\in\overline{\mathcal{O}}$ such that for each $x\in X$, the set $\{U\in\mathcal{U} : x\notin\overline{U}\}$ is finite.

\item[$\overline{\Lambda}$:] The family of all sets $\mathcal{U}\in\overline{\mathcal{O}}$ such that for each $x\in X$, the set $\{U\in\mathcal{U} : x\in\overline{U}\}$ is infinite.

\item[$\overline{\mathcal{O}^{wgp}}$:] The family of all sets $\mathcal{U}\in\overline{\mathcal{O}}$ such that $\mathcal{U}$ can be expressed as a countable union of finite, pairwise disjoint subfamilies $\mathcal{U}_n$, $n\in\mathbb{N}$, such that for each finite set $F\subseteq X$ there exists a $n$ such that $F\subseteq\overline{\cup\mathcal{U}_n}$.

\item[$\overline{\Lambda^{wgp}}$:] The family of all sets $\mathcal{U}\in\overline{\Lambda}$ such that $\mathcal{U}\in\overline{\mathcal{O}^{wgp}}$.

\item[$\mathcal{O}_D$:] The family of all sets $\mathcal{U}$ of open subsets of $X$ such that $\cup\mathcal{U}$ is dense in $X$.

\item[$\Omega_D$:] The family of all sets $\mathcal{U}\in\mathcal{O}_D$ such that for each finite collection $\mathcal{F}$ of nonempty open sets of $X$ there exists a $U\in\mathcal{U}$ such that $U\cap V\neq\emptyset$ for all $V\in\mathcal{F}$.

\item[$\Gamma_D$:] The family of all sets $\mathcal{U}\in\mathcal{O}_D$ such that for each nonempty open set $U\subseteq X$, the set $\{V\in\mathcal{U} : U\cap V=\emptyset\}$ is finite.

\item[$\Lambda_D$:] The family of all sets $\mathcal{U}\in\mathcal{O}_D$ such that for each nonempty open set $U\subseteq X$, the set $\{V\in\mathcal{U} : U\cap V\neq\emptyset\}$ is infinite.

\item[${\mathcal{O}^{wgp}}_D$:] The family of all sets $\mathcal{U}\in\mathcal{O}_D$ such that $\mathcal{U}$ can be expressed as a countable union of finite, pairwise disjoint subfamilies $\mathcal{U}_n$, $n\in\mathbb{N}$, such that for each finite collection $\mathcal{F}$ of nonempty open sets of $X$ there exists a $n$ such that $V\cap(\cup\mathcal{U}_n)\neq\emptyset$ for all $V\in\mathcal{F}$.

\item[${\Lambda^{wgp}}_D$:] The family of all sets $\mathcal{U}\in\Lambda_D$ such that $\mathcal{U}\in{\mathcal{O}^{wgp}}_D$.

\item[$\Omega^D$:] The family of all sets $\mathcal{U}\in\Omega_D$ such that for each $\mathcal{U}$ there exists a dense set $Y\subseteq X$ such that each finite set $F\subseteq Y$ is contained in $U$ for some $U\in\mathcal{U}$.

\item[$\Gamma^D$:] The family of all sets $\mathcal{U}\in\Gamma_D$ such that for each $\mathcal{U}$ there exists a dense set $Y\subseteq X$ such that for each $x\in Y$, the set $\{U\in\mathcal{U} : x\notin U\}$ is finite.

\item[$\Lambda^D$:] The family of all sets $\mathcal{U}\in\Lambda_D$ such that for each $\mathcal{U}$ there exists a dense set $Y\subseteq X$ such that for each $x\in Y$, the set $\{U\in\mathcal{U} : x\in U\}$ is infinite.

\item[${\mathcal{O}^{wgp}}^D$:] The family of all sets $\mathcal{U}\in{\mathcal{O}^{wgp}}_D$ such that for each $\mathcal{U}$ there exists a dense set $Y\subseteq X$ and $\mathcal{U}$ can be expressed as a countable union of finite, pairwise disjoint subfamilies $\mathcal{U}_n$, $n\in\mathbb{N}$, such that each finite set $F\subseteq Y$ is contained in $\cup\mathcal{U}_n$ for some $n$.

\item[${\Lambda^{wgp}}^D$:] The family of all sets $\mathcal{U}\in\Lambda^D$ such that $\mathcal{U}\in{\mathcal{O}^{wgp}}^D$.
\end{enumerate}

All the considered covers are assumed to be infinite. Observe that
\begin{multicols}{2}
\begin{enumerate}[label={\upshape(\arabic*)}]
  \item $\overline{\Gamma}\subseteq\overline{\Omega}\subseteq\overline{\Lambda}
      \subseteq\overline{\mathcal{O}}$
  \item $\Gamma_D\subseteq\Omega_D\subseteq\Lambda_D\subseteq\mathcal{O}_D$
  \item $\Gamma^D\subseteq\Omega^D\subseteq\Lambda^D\subseteq\mathcal{O}_D$
  \item $\overline{\mathcal{O}}\subseteq\mathcal{O}_D$
  \item $\overline{\Gamma}\subseteq\Gamma_D\subseteq\Gamma^D$
  \item $\overline{\Omega}\subseteq\Omega_D\subseteq\Omega^D$
  \item $\overline{\Lambda}\subseteq\Lambda_D\subseteq\Lambda^D$.
\end{enumerate}
\end{multicols}

Also note that every countable member of $\overline{\Omega}$ (respectively, $\Omega_D$, $\Omega^D$) is a member of $\overline{\mathcal{O}^{wgp}}$ (respectively, ${\mathcal{O}^{wgp}}_D$, ${\mathcal{O}^{wgp}}^D$).

A space $X$ is said to be almost Lindel\"{o}f (in short, $\aL$) (respectively, weakly Lindel\"{o}f (in short, $\wL$)) if for every open cover $\mathcal{U}$ of $X$ there exists a countable subset $\mathcal{V}$ of $\mathcal{U}$ such that $\mathcal{V}\in\overline{\mathcal{O}}$ (respectively, $\mathcal{V}\in\mathcal{O}_D$) (see \cite{ALWL,MTCP}).
A space $X$ satisfying the selection principle $\Sf(\mathcal{O},\overline{\mathcal{O}})$ (respectively, $\Uf(\mathcal{O},\overline{\Gamma})$) is called almost Menger \cite{SMR-II} (in short, $\aM$) (respectively, almost Hurewicz \cite{AHS,SHS} (in short, $\aH$)).
A space $X$ is said to be weakly Menger \cite{PRS,WMP} (in short, $\wM$) (respectively, weakly Hurewicz in the sense of Sakai \cite{WHP} (in short, $\wH$)) if $X$ satisfies the selection principle $\Sf(\mathcal{O},\mathcal{O}_D)$ (respectively, $\Uf(\mathcal{O},\Gamma_D)$).
A space $X$ is said to be weakly Hurewicz in the sense of Ko\v{c}inac \cite{PRTS} (in short, $\wHk$) if $X$ satisfies the selection principle $\Uf(\mathcal{O},\Gamma^D)$.

The following terminologies will be used throughout our study (for similar such notions, see \cite{WCPSP,Alster}).
\begin{enumerate}[leftmargin=*]
\item[$\mathcal{G}$:] The family of all covers $\mathcal{U}$ of the space $X$ for which each element of $\mathcal{U}$ is a $G_\delta$ set.

\item[$\mathcal{G}_K$:] The family of all sets $\mathcal{U}$ where $X$ is not in $\mathcal{U}$, each element of $\mathcal{U}$ is a $G_\delta$ set, and for each compact set $C\subseteq X$ there is a $U\in\mathcal{U}$ such that $C\subseteq U$.

\item[$\mathcal{G}_\Omega$:] The family of all covers $\mathcal{U}\in\mathcal{G}$ such that for each finite set $F\subseteq X$ there is a $U\in\mathcal{U}$ such that $F\subseteq U$.

\item[$\mathcal{G}_\Gamma$:] The family of all covers $\mathcal{U}\in\mathcal{G}$ which are infinite and each infinite subset of $\mathcal{U}$ is a cover of $X$.

\item[$\mathcal{G}_{\overline{\Gamma}}$:] The family of all covers $\mathcal{U}\in\mathcal{G}$ which are infinite and for each $x\in X$, the set $\{U\in\mathcal{U} : x\notin\overline{U}\}$ is finite.

\item[$\mathcal{G}_D$:] The family of all sets $\mathcal{U}$ where each element of $\mathcal{U}$ is a $G_\delta$ set and $\cup\mathcal{U}$ is dense in $X$.

\item[$\mathcal{G}_{D_\Gamma}$:] The family of all sets $\mathcal{U}$ where each element of $\mathcal{U}$ is a $G_\delta$ set and for each nonempty open set $U\subseteq X$, the set $\{V\in\mathcal{U} : U\cap V=\emptyset\}$ is finite.

\item[$\mathcal{G}_{\Gamma_D}$:] The family of all sets $\mathcal{U}$ where each element of $\mathcal{U}$ is a $G_\delta$ set and for each $\mathcal{U}$ there exists a dense set $Y\subseteq X$ such that for each $x\in Y$, the set $\{U\in\mathcal{U} : x\notin U\}$ is finite.

\item[$\overline{\mathcal{G}}$:] The family of all sets $\mathcal{U}$ such that every $U\in\mathcal{U}$ is a $G_\delta$ set and $\{\overline{U} : U\in\mathcal{U}\}$ covers $X$.

\item[$\overline{\mathcal{G}_\Omega}$:] The family of all sets $\mathcal{U}\in\overline{\mathcal{G}}$ such that for each finite set $F\subseteq X$ there exists a $U\in\mathcal{U}$ such that $F\subseteq U$.
\end{enumerate}

A space $X$ is said to be Alster \cite{Alster} if each member $\mathcal{U}$ of $\mathcal{G}_K$ has a countable subset $\mathcal{V}$ that covers $X$ (or, equivalently if $X$ satisfies the selection principle $\S1(\mathcal{G}_K,\mathcal{G})$).
A space $X$ is said to be almost Alster \cite{MTCP} if each member $\mathcal{U}$ of $\mathcal{G}_K$ has a countable subset $\mathcal{V}$ such that $\{\overline{V} : V\in\mathcal{V}\}$ covers $X$ (or, equivalently if $X$ satisfies the selection principle $\S1(\mathcal{G}_K,\overline{\mathcal{G}})$).
A space $X$ is said to be weakly Alster \cite{WCPSP} if each member $\mathcal{U}$ of $\mathcal{G}_K$ has a countable subset $\mathcal{V}$ such that $\cup\mathcal{V}$ is dense in $X$ (or, equivalently if $X$ satisfies the selection principle $\S1(\mathcal{G}_K,\mathcal{G}_D)$).

Consider the Baire space $\mathbb{N}^\mathbb{N}$. A natural pre-order $\leq^*$ on $\mathbb{N}^\mathbb{N}$ is defined by $f\leq^*g$ if and only if $f(n)\leq g(n)$ for all but finitely many $n$. A subset $D$ of $\mathbb{N}^\mathbb{N}$ is said to be dominating if for each $g\in\mathbb{N}^\mathbb{N}$ there exists a $f\in D$ such that $g\leq^* f$. A subset $A$ of $\mathbb{N}^\mathbb{N}$ is said to be bounded if there is a $g\in\mathbb{N}^\mathbb{N}$ such that $f\leq^*g$ for all $f\in A$. Let $\mathfrak{d}$ be the minimum cardinality of a dominating subset of $\mathbb{N}^\mathbb{N}$, $\mathfrak{b}$ be the minimum cardinality of an unbounded subset of $\mathbb{N}^\mathbb{N}$ and $\mathfrak{c}$ be the cardinality of the continuum.

A subset $A$ of a space $X$ is said to be regular-closed (respectively, regular-open) if $\overline{\Int(A)}=A$ (respectively, $\Int(\overline{A})=A$). Every clopen set is regular-closed as well as regular-open.
The Alexandroff duplicate \cite{AD,Engelking} $AD(X)$ of a space $X$ is defined as follows. $AD(X)=X\times\{0,1\}$; each point of $X\times\{1\}$ is isolated and a basic neighbourhood of $(x,0)\in X\times\{0\}$ is a set of the form $(U\times\{0\})\cup((U\times\{1\})\setminus\{(x,1)\})$, where $U$ is a neighbourhood of $x$ in $X$.
For a space $X$, $\PR(X)$ denotes the space of all nonempty finite subsets of $X$ with the Pixley-Roy topology \cite{PRS}. The collection $\{[A,U] : A\in \PR(X),\; U\;\text{open in}\; X\}$ is a base for the Pixley-Roy topology, where $[A,U]=\{B\in \PR(X) : A\subseteq B\subseteq U\}$ for each $A\in \PR(X)$ and each open set $U$ in $X$.
A space $X$ is said to be cosmic if it has a countable network. A space $X$ satisfies the countable chain condition (in short, CCC) if every family of disjoint nonempty open subsets of $X$ is countable. A space satisfying the CCC is called a CCC space. For a space $X$, $e(X)=\sup\{|Y| : Y\;\text{is a closed and discrete subspace of}\; X\}$ is said to be the extent of $X$. For any two spaces $X$ and $Y$, $X\oplus Y$ denote the topological sum of $X$ and $Y$.
\begin{Lemma}
For a space $X$ the following assertions hold.
\begin{enumerate}[wide=0pt,label={\upshape(\arabic*)},
leftmargin=*,ref={\theLemma(\arabic*)}]
  \item\label{L2} $\Sf(\Gamma,\overline{\Lambda})=\Sf(\Omega,\overline{\Lambda})$
  \item\label{L6} $\Sf(\Gamma,\Lambda_D)=\Sf(\Omega,\Lambda_D)$
  \item\label{L8} $\Sf(\Gamma,\Lambda^D)=\Sf(\Omega,\Lambda^D)$.
\end{enumerate}
\end{Lemma}
\begin{proof}
We only give proof for the first assertion.
Let $(\mathcal{U}_n)$ be a sequence of $\omega$-covers of $X$. Without loss of generality we may assume that for each finite set $\mathcal{F}\subseteq\cup_{n\in\mathbb{N}}\mathcal{U}_n$, $\mathcal{U}_k\cap\mathcal{F}=\emptyset$ for all but finitely many $k$. For each $n$ enumerate $\mathcal{U}_n$ bijectively as $\{U_m^{(n)} : m\in\mathbb{N}\}$. Next for each $n$ and each $m$ define $V_m^{(n)}=\cup_{1\leq i\leq m}U_i^{(n)}$. Clearly $\mathcal{W}_n=\{V_m^{(n)} : m\in\mathbb{N}\}$ is a $\gamma$-cover of $X$ for each $n$. Using the hypothesis we can find a sequence $(\mathcal{H}_n)$ such that for each $\mathcal{H}_n$ is a finite subset of $\mathcal{W}_n$ and $\cup_{n\in\mathbb{N}}\mathcal{H}_n\in\overline{\Lambda}$. Clearly $(\mathcal{H}_n)$ produces a sequence $(\mathcal{V}_n)$ such that for each $n$ $\mathcal{V}_n$ is a finite subset of $\mathcal{U}_n$ and $\cup\mathcal{V}_n=\cup\mathcal{H}_n$. Since each $\mathcal{V}_n$ is disjoint from $\mathcal{U}_k$ for all but finitely many $k$, $\cup_{n\in\mathbb{N}}\mathcal{V}_n\in\overline{\Lambda}$. Thus $X$ has the property $\Sf(\Omega,\overline{\Lambda})$.
\end{proof}

\section{Weaker versions of the Scheepers property}
\subsection{Almost Scheepers and related spaces}
We introduce the following definitions.
\begin{Def}\hfill
\begin{enumerate}[wide=0pt,label={\upshape(\arabic*)},leftmargin=*]
\item \label{D1}
A space $X$ is said to be almost Scheepers (in short, $\aS$) if $X$ satisfies the selection principle $\Uf(\mathcal{O},\overline{\Omega})$ i.e. if for every sequence $(\mathcal{U}_n)$ of open covers of $X$ there exists a sequence $(\mathcal{V}_n)$ such that for each $n$ $\mathcal{V}_n$ is a finite subset of $\mathcal{U}_n$ and $\{\cup\mathcal{V}_n : n\in\mathbb{N}\}\in\overline{\Omega}$.
\item \label{D2}
A space $X$ is said to be weakly Scheepers (in short, $\wS$) if $X$ satisfies the selection principle $\Uf(\mathcal{O},\Omega_D)$.
\end{enumerate}
\end{Def}

We also introduce another weak variation of the Scheepers property.
\begin{Def}
\label{D3}
A space $X$ is said to be weakly Scheepers in the sense of Ko\v{c}inac (in short, $\wSk$) if $X$ satisfies the selection principle $\Uf(\mathcal{O},\Omega^D)$.
\end{Def}

The following equivalent formulations of the weaker forms of the Scheepers property will be used subsequently. A space $X$ is
\begin{enumerate}[wide=0pt,label={\upshape(\arabic*)},leftmargin=*]
\item $\aS$ if for every sequence $(\mathcal{U}_n)$ of open covers of $X$ there exists a sequence $(\mathcal{V}_n)$ such that for each $n$ $\mathcal{V}_n$ is a finite subset of $\mathcal{U}_n$ and each finite set $F\subseteq X$ is contained in $\overline{\cup\mathcal{V}_n}$ for infinitely many $n$.
\item $\wS$ if for every sequence $(\mathcal{U}_n)$ of open covers of $X$ there exists a sequence $(\mathcal{V}_n)$ such that for each $n$ $\mathcal{V}_n$ is a finite subset of $\mathcal{U}_n$ and for each finite collection $\mathcal{F}$ of nonempty open subsets of $X$, the set $\{n\in\mathbb{N} : U\cap(\cup\mathcal{V}_n)\neq\emptyset\;\text{for all}\;U\in\mathcal{F}$\} is infinite.
\item $\wSk$ if for every sequence $(\mathcal{U}_n)$ of open covers of $X$ there exist a dense subset $Y$ of $X$ and a sequence $(\mathcal{V}_n)$ such that for each $n$ $\mathcal{V}_n$ is a finite subset of $\mathcal{U}_n$ and each finite set $F\subseteq Y$ is contained in $\cup\mathcal{V}_n$ for infinitely many $n$.
\end{enumerate}

The relation between the weaker forms of the Hurewicz ({\sf H}), Scheepers ({\sf S}), Menger ({\sf M}) and Lindel\"{o}f ({\sf L}) properties are outlined in the following implication diagram (Figure~\ref{dig1}). It is easy to observe that under regularity condition, $\aS$ implies the Scheepers property and similarly for the other `almost' variations. Interestingly all the weaker versions are hereditary for clopen subsets. We will come back again with preservation like observations in the next section.
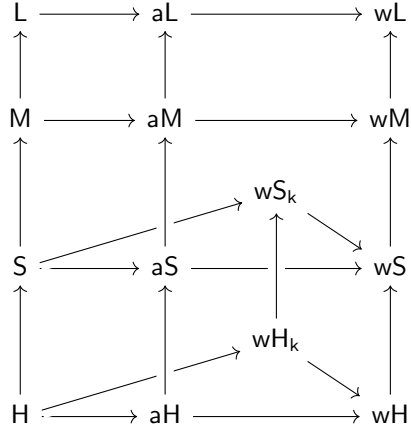
\begin{figure}[h]
\begin{adjustbox}{max width=\textwidth,max height=\textheight,keepaspectratio,center}
\begin{tikzcd}[column sep=4ex,row sep=3ex,arrows={crossing over}]
{\sf L}\arrow[rr]&&\aL\arrow[rr]&&\wL
\\
\\
{\sf M}\arrow[rr]\arrow[uu] && \aM\arrow[rr]\arrow[uu] && \wM\arrow[uu]
\\
&&&\wSk\arrow[rd]&
\\
{\sf S}\arrow[rr]\arrow[uu]\arrow[rrru]&&
\aS\arrow[rr]\arrow[uu]&&
\wS\arrow[uu]
\\
&&&\wHk\arrow[rd]\arrow[uu]&
\\
{\sf H}\arrow[rr]\arrow[uu]\arrow[rrru]&&
\aH\arrow[rr]\arrow[uu]&&
\wH\arrow[uu]
\end{tikzcd}
\end{adjustbox}
\caption{Weaker forms of the Hurewicz, Scheepers, Menger and Lindel\"{o}f properties}
\label{dig1}
\end{figure}

We now illustrate to distinguish these weaker forms of the classical selective properties.
\begin{Ex}
\label{E1}
\emph{There exists a $\aS$ space which is not Scheepers.}\\
Let $A=\{(a_\alpha,-1) : \alpha<\omega_1\}\subseteq \{(x,-1) : x\geq 0\}\subseteq\mathbb{R}^2$, $Y=\{(a_\alpha,n) : \alpha<\omega_1, n\in\mathbb{N}\}$ and $p=(-1,-1)$. Choose $X=Y\cup A\cup\{p\}$. We topologize $X$ as follows. $(i)$ Every point of $Y$ is isolated, $(ii)$ for each $\alpha<\omega_1$ a basic neighbourhood of $(a_\alpha,-1)$ is of the form $U_n(a_\alpha,-1)=\{(a_\alpha,-1)\}\cup\{(a_\alpha,m) : m\geq n\}$, where $n\in\mathbb{N}$ and $(iii)$ a basic neighbourhood of $p$ is of the form $U_\alpha(p)=\{p\}\cup\{(a_\beta,n) : \beta>\alpha, n\in\mathbb{N}\}$, where $\alpha<\omega_1$. Clearly $X$ is Urysohn. Since $A$ is an uncountable discrete closed subset of $X$, $X$ is not Lindel\"{o}f and hence not Scheepers.

We now show that $X$ is $\aS$. Let $(\mathcal{U}_n)$ be a sequence of open covers of $X$. For each $n$ choose a $U_n\in\mathcal{U}_n$ such that $p\in U_n$ and also choose a basic neighbourhood $U_{\beta_n}(p)\subseteq U_n$. Now for each $n$ $X\setminus\overline{U_n}$ is at most countable as $\overline{U_{\beta_n}(p)}=U_{\beta_n}(p)\cup\{(a_\beta,-1) : \beta>\beta_n\}$. Clearly $Y=\cup_{n\in\mathbb{N}}(X\setminus\overline{U_n})$ is $\sigma$-compact and hence Hurewicz. Apply the Hurewicz property of $Y$ to $(\mathcal{U}_n)$ to obtain a sequence $(\mathcal{V}_n^\prime)$ such that for each $n$ $\mathcal{V}_n^\prime$ is a finite subset of $\mathcal{U}_n$ and each $x\in Y$ belongs to $\cup\mathcal{V}_n^\prime$ for all but finitely many $n$. Clearly the sequence $(\mathcal{V}_n)$, where $\mathcal{V}_n=\mathcal{V}_n^\prime\cup\{U_n\}$ for each $n$ witnesses that $X$ is $\aH$ (and hence $\aS$).
\end{Ex}

For a Tychonoff space $X$, let $\beta X$ denote the Stone-\v{C}ech compactification of $X$.
\begin{Ex}
\label{E2}
\emph{There exists a $\wSk$ (and also a $\wS$) space which is not $\aS$ (and hence not Scheepers).}\\
Let $D$ be the discrete space of cardinality $\omega_1$. Consider $X=(\beta D\times (\omega+1))\setminus((\beta D\setminus D)\times\{\omega\})$ as a subspace of $\beta D\times(\omega+1)$. Since $\beta D\times\omega$ is a $\sigma$-compact dense subset of $X$, $X$ is $\wSk$ (see Section~\ref{S1} below). Now $X$ is not $\aL$ (see \cite[Example 2.3]{ALWL}) and so $X$ is not $\aS$.
\end{Ex}

The existence of a $\aS$ space which is not $\aH$ follows from the existence of a set of reals which is Scheepers but not Hurewicz (see \cite[Theorem 8.10]{SFPH}). By \cite[Theorem 2.1]{FPM}, there exists a set of reals which is Menger but not Scheepers. This shows that there is a $\aM$ space which is not $\aS$.

It is well known that every CCC space is $\wL$  \cite{ALWL,PRT1}. A CCC space need not satisfy any of the weaker variations of the Scheepers property.
\begin{Ex}
\label{E9}\emph{There is a CCC space which is neither $\aS$ nor $\wS$ nor $\wSk$.}\\
Let $Z=\PR(\mathbb{P})$, where $\mathbb{P}$ is the space of irrationals. Clearly $\mathbb{P}$ is a cosmic space which is not Menger. It has been observed that $\PR(X)$ is CCC for every regular cosmic space $X$ \cite{PRT1} and also if $\PR(X)$ is $\wM$, then each finite power of $X$ is Menger  \cite[Theorem 2A]{PRS}. Thus $Z$ is a CCC space which is not $\wM$. Which shows that $Z$ is not $\wS$ and hence $Z$ is neither $\aS$ nor $\wSk$.
\end{Ex}

We now reconfigure the above defined notions using regular-open sets and neighbourhood assignments.
\begin{Prop}
\hfill
\begin{enumerate}[wide=0pt,label={\upshape(\arabic*)},leftmargin=*]
\item \label{T1}
A space $X$ is $\aS$ if and only if for every sequence $(\mathcal{U}_n)$ of covers of $X$ by regular-open sets there is a sequence $(\mathcal{V}_n)$ such that for each $n$ $\mathcal{V}_n$ is a finite subset of $\mathcal{U}_n$ and for each finite set $F\subseteq X$ there is a $n$ such that $F\subseteq\overline{\cup\mathcal{V}_n}$.
\item\label{T2}
A space $X$ is $\wS$ if and only if for every sequence $(\mathcal{U}_n)$ of covers of $X$ by regular-open sets there exists a sequence $(\mathcal{V}_n)$ such that for each $n$ $\mathcal{V}_n$ is a finite subset of $\mathcal{U}_n$ and for each finite collection $\mathcal{F}$ of nonempty regular-open sets there is a $n$ such that $U\cap(\cup\mathcal{V}_n)\neq\emptyset$ for all $U\in\mathcal{F}$.
\item \label{T27}
Suppose that $X$ is regular. Then $X$ is $\wSk$ if and only if for every sequence $(\mathcal{U}_n)$ of covers of $X$ by regular-open sets there exist a dense subset $Y$ of $X$ and a sequence $(\mathcal{V}_n)$ such that for each $n$ $\mathcal{V}_n$ is a finite subset of $\mathcal{U}_n$ and for each finite set $F\subseteq Y$ there is a $n$ such that $F\subseteq\cup\mathcal{V}_n$.
\end{enumerate}
\end{Prop}

Recall that a neighbourhood assignment for a space $(X,\tau)$ is a function $\n:X\to\tau$ such that $x\in \n(x)$ for each $x\in X$.
\begin{Th}
\label{T52}
\hfill
\begin{enumerate}[wide=0pt,label={\upshape(\arabic*)},leftmargin=*]
  \item A space $X$ is $\aS$ if and only if for each sequence $(\n_n)$ of neighbourhood assignments of $X$ there exists a sequence $(F_n)$ of finite subsets of $X$ such that $\{\n(F_n) : n\in\mathbb{N}\}\in\overline{\Omega}$.
  \item A space $X$ is $\wS$ if and only if for each sequence $(\n_n)$ of neighbourhood assignments of $X$ there exists a sequence $(F_n)$ of finite subsets of $X$ such that $\{\n(F_n) : n\in\mathbb{N}\}\in\Omega_D$.
  \item A space $X$ is $\wSk$ if and only if for each sequence $(\n_n)$ of neighbourhood assignments of $X$ there exists a sequence $(F_n)$ of finite subsets of $X$ such that $\{\n(F_n) : n\in\mathbb{N}\}\in\Omega^D$.
\end{enumerate}
\end{Th}
\begin{proof}
We only give proof for $(1)$.
Let $(\n_n)$ be a sequence of neighbourhood assignments of $X$. For each $n$ we define $\mathcal{U}_n=\{\n_n(x) : x\in X\}$. Then $(\mathcal{U}_n)$ is a sequence of open covers of $X$ and since $X$ is $\aS$, there exists a sequence $(\mathcal{V}_n)$ such that for each $n$ $\mathcal{V}_n$ is a finite subset of $\mathcal{U}_n$ and $\{\cup\mathcal{V}_n : n\in\mathbb{N}\}\in\overline{\Omega}$. This gives us a sequence $(F_n)$ of finite subsets of $X$ such that $\mathcal{V}_n=\{\n_n(x) : x\in F_n\}$ for each $n$. It follows that $\{\n_n(F_n) : n\in\mathbb{N}\}\in\overline{\Omega}$.

Conversely let $(\mathcal{U}_n)$ be a sequence of open covers of $X$. For each $x\in X$ there exists a $U_x^{(n)}\in\mathcal{U}_n$ such that $x\in U_x^{(n)}$. Consider the sequence $(\n_n)$ of neighbourhood assignments of $X$, where $\n_n(x)=U_x^{(n)}$ for each $n$ and for each $x\in X$. By the given hypothesis, there is a sequence $(F_n)$ of finite subsets of $X$ such that $\{\n_n(F_n) : n\in\mathbb{N}\}\in\overline{\Omega}$. For each $n$ choose $\mathcal{V}_n=\{\n_n(x) : x\in F_n\}$. Clearly the sequence $(\mathcal{V}_n)$ witnesses for $(\mathcal{U}_n)$ that $X$ is $\aS$.
\end{proof}

The following result can be easily verified.
\begin{Th}
\label{T26}
Every paracompact $\aS$ space is Scheepers.
\end{Th}

An open cover $\mathcal{U}$ of a space $X$ is called star finite if every $U\in\mathcal{U}$ intersects only finitely many $V\in\mathcal{U}$. A space $X$ is called hypocompact if every open cover $\mathcal{U}$ of $X$ has a star finite open refinement. Since every hypocompact space is paracompact, we obtain the following.

\begin{Cor}
\label{C18}
Every hypocompact $\aS$ space is Scheepers.
\end{Cor}

The above result does not hold in the context of $\wS$ and $\wSk$ spaces. The Baire space is a typical counter-example to it (see Remark~\ref{R1}).

\begin{Th}
\label{T5}\hfill
\begin{enumerate}[wide=0pt,label={\upshape(\arabic*)}]
\item If every finite power of a space $X$ is $\aM$, then $X$ is $\aS$.
\item If every finite power of a space $X$ is $\wM$, then $X$ is $\wS$.
\end{enumerate}
\end{Th}

\begin{proof}
We only give proof of (2).
Let $(\mathcal{U}_n)$ be a sequence of open covers of $X$. Let $\{N_k : k\in\mathbb{N}\}$ be a partition of $\mathbb{N}$ into infinite subsets. For each $k$ and each $m\in N_k$ let $\mathcal{W}_m=\{U_1\times U_2\times\dotsm\times U_k : U_1,U_2,\dotsc,U_k\in\mathcal{U}_m\}$. Clearly $(\mathcal{W}_m : m\in N_k)$ is a sequence of open covers of $X^k$. Apply the $\wM$ property of $X^k$ to $(\mathcal{W}_m : m\in N_k)$ to obtain a sequence $(\mathcal{H}_m : m\in N_k)$ of finite sets such that $\mathcal{H}_m\subseteq \mathcal{W}_m$ for each $m\in N_k$ and $\cup\{\cup\mathcal{H}_m : m\in N_k\}$ is dense in $X^k$. For each $m\in N_k$  we can express every $H\in\mathcal{H}_m$ as $H=U_1(H)\times U_2(H)\times\dotsm\times U_k(H)$, where $U_1(H),U_2(H),\dotsc,U_k(H)\in\mathcal{U}_m$. Choose $\mathcal{V}_m=\{U_i(H) : 1\leq i\leq k, H\in\mathcal{H}_m\}$. Clearly $\mathcal{V}_m$ is a finite subset of $\mathcal{U}_m$ for each $m\in N_k$. Thus we obtain a sequence $(\mathcal{V}_n)$ of finite sets with $\mathcal{V}_n\subseteq\mathcal{U}_n$. To complete the proof, choose  a finite collection $\mathcal{F}=\{U_1,U_2,\dotsc,U_p\}$ of nonempty open subsets of $X$. Since $U_1\times U_2\times\dotsm\times U_p$ is a nonempty open set in $X^p$, there is a $m_0\in N_p$ such that $(U_1\times U_2\times\dotsm\times U_p)\cap(\cup\mathcal{H}_{m_0})\neq\emptyset$. Which in turn implies that $U_i\cap(\cup\mathcal{V}_{m_0})\neq\emptyset$ for each $1\leq i\leq p$. Clearly $(\mathcal{V}_n)$ witnesses that $X$ is $\wS$.
\end{proof}

For any families $\mathcal{U}$ and $\mathcal{V}$ of subsets of $X$ we denote the set $\{U\cap V : U\in\mathcal{U}\;\text{and}\; V\in\mathcal{V}\}$ by $\mathcal{U}\wedge\mathcal{V}$.
\begin{Th}
\label{T45}
For a space $X$ the following assertions are equivalent.
\begin{enumerate}[wide=0pt,label={\upshape(\arabic*)},leftmargin=*]
  \item $X$ is $\aS$.
  \item $X$ satisfies $\Uf(\mathcal{O},\overline{\mathcal{O}^{wgp}})$.
\end{enumerate}
\end{Th}
\begin{proof}
$(1)\Rightarrow (2)$ follows from the fact that every countable member of $\overline{\Omega}$ is also a member of $\overline{\mathcal{O}^{wgp}}$. For the other implication, choose a sequence $(\mathcal{U}_n)$ of open covers of $X$. Consider the sequence $(\mathcal{W}_n)$ of open covers of $X$, where $\mathcal{W}_n=\wedge_{i\leq n}\mathcal{U}_n$ for each $n$. Since $X$ satisfies $\Uf(\mathcal{O},\overline{\mathcal{O}^{wgp}})$, there exists a sequence $(\mathcal{H}_n)$ such that for each $n$ $\mathcal{H}_n$ is a finite subset of $\mathcal{W}_n$ and $\{\cup\mathcal{H}_n : n\in\mathbb{N}\}\in\overline{\mathcal{O}^{wgp}}$. Choose a sequence $n_1<n_2<\cdots$ of positive integers witnessing $\{\cup\mathcal{H}_n : n\in\mathbb{N}\}\in\overline{\mathcal{O}^{wgp}}$ i.e. each finite set $F\subseteq X$ is contained in $\overline{\cup\{\cup\mathcal{H}_i : n_k\leq i<n_{k+1}\}}$ for some $k\in\mathbb{N}$. For each $n$ define \[\mathcal{K}_n=
\begin{cases}
  \cup_{i<n_1}\mathcal{H}_i, & \mbox{if }\; n<n_1 \\
  \cup_{n_k\leq i<n_{k+1}}\mathcal{H}_i, & \mbox{if}\; n_k\leq n<n_{k+1}.
\end{cases}\]

Now define a sequence $(\mathcal{V}_n)$ as follows. For each $n$ $\mathcal{V}_n$ is the collection of all members of $\mathcal{U}_n$ in the representation of each member of $\mathcal{K}_n$. Clearly for each $n$ $\mathcal{V}_n$ is a finite subset of $\mathcal{U}_n$ and $\cup\mathcal{K}_n\subseteq\cup\mathcal{V}_n$. The sequence $(\mathcal{V}_n)$ witnesses that $X$ is $\aS$.
\end{proof}

The next two results can be similarly verified.
\begin{Th}
\label{T46}
For a space $X$ the following assertions are equivalent.
\begin{enumerate}[wide=0pt,label={\upshape(\arabic*)},leftmargin=*]
  \item $X$ is $\wS$.
  \item $X$ satisfies $\Uf(\mathcal{O},{\mathcal{O}^{wgp}}_D)$.
\end{enumerate}
\end{Th}

\begin{Th}
\label{T47}
For a space $X$ the following assertions are equivalent.
\begin{enumerate}[wide=0pt,label={\upshape(\arabic*)},leftmargin=*]
  \item $X$ is $\wSk$.
  \item $X$ satisfies $\Uf(\mathcal{O},{\mathcal{O}^{wgp}}^D)$.
\end{enumerate}
\end{Th}
%

Note that if a space $X$ has the Alster property, then $X$ is Scheepers. A similar observation for the weaker versions is presented in the next result.


\begin{Th}
\label{T43}
\hfill
\begin{enumerate}[wide=0pt,label={\upshape(\arabic*)},
ref={\theTh(\arabic*)},leftmargin=*]
  \item\label{T4301} Every almost Alster space is $\aS$.
  \item\label{T4302} Every weakly Alster space is $\wS$.
\end{enumerate}
\end{Th}
\begin{proof}
We only provide proof for the `almost version'.
Let $X$ be almost Alster. To show that $X$ is $\aS$ we choose a sequence $(\mathcal{U}_n)$ of open covers of $X$. We may assume that for each $n$ $\mathcal{U}_n$ is closed for finite unions. Let $\{N_k : k\in\mathbb{N}\}$ be a partition of $\mathbb{N}$ into infinite sets. For each $k$ and each $n\in N_k$ choose $\mathcal{W}_n=\{U^k : U\in\mathcal{U}_n\}$. Thus for each $k$ $(\mathcal{W}_n : n\in N_k)$ is a sequence of open covers of $X^k$. Fix $k$. Let $\mathcal{U}=\{\cap_{n\in N_k}W_n : W_n\in\mathcal{W}_n\}$. Obviously $\mathcal{U}\in\mathcal{G}_K$ for $X^k$. Without loss of generality we suppose that $\mathcal{U}=\mathcal{H}_1\times\mathcal{H}_2\times\cdots\times\mathcal{H}_k$, where each $\mathcal{H}_i\in\mathcal{G}_K$ for $X$. Using the almost Alster property of $X$, for each $1\leq i\leq k$, we obtain a countable set $\mathcal{C}_i\subseteq\mathcal{H}_i$ such that $\{\overline{G} : G\in\mathcal{C}_i\}$ covers $X$ and hence there is a countable set $\mathcal{V}=\mathcal{C}_1\times\mathcal{C}_2\times\cdots\times\mathcal{C}_k\subseteq\mathcal{U}$ such that $\{\overline{V} : V\in\mathcal{V}\}$ covers $X^k$. Say $\mathcal{V}=\{\cap_{n\in N_k}W_n^{(m)} : W_n^{(m)}\in\mathcal{W}_n, m\in N_k\}$. Choose $\mathcal{V}^\prime=\{W_n^{(n)}\in\mathcal{W}_n : n\in N_k \}$. Now $\cup\mathcal{V}\subseteq\cup\mathcal{V}^\prime$ and consequently $\{\overline{V} : V\in\mathcal{V}^\prime\}$ covers $X^k$. For each $n\in N_k$ let $\mathcal{V}_n=\{U\in\mathcal{U}_n : U^k\in\mathcal{V}^\prime\}$. The sequence $(\mathcal{V}_n)$ witnesses that $X$ is $\aS$.

\end{proof}

In the context of paracompact spaces as well as regular spaces, the $\aS$ property and the Scheepers property are equivalent. In \cite[Proposition 2.3]{MTCP}, the authors answered the question \cite[Question 13]{WMP} that under what conditions the $\aM$ and $\wM$ properties are equivalent. We now present Theorem~\ref{T44}, which is an improvement of \cite[Proposition 2.3]{MTCP}. First we need the following lemma which is itself a refinement of \cite[Lemma 55]{WCPSP}, \cite[Proposition 2.3]{MTCP} and \cite[Proposition 2.4]{MTCP}.
\begin{Lemma}
\label{L7}
Every $\wL$ $P$-space $X$ is almost Alster.
\end{Lemma}
\begin{proof}
Let $\mathcal{U}\in\mathcal{G}_K$. Since $X$ is a $P$-space, $\mathcal{U}$ is an open cover of $X$. Also since $X$ is $\wL$, there is a countable set $\mathcal{V}\subseteq\mathcal{U}$ such that $\cup\mathcal{V}$ is dense in $X$. Again by the property of a $P$-space we can say that the set $\cup\{\overline{V} : V\in\mathcal{V}\}$ is closed in $X$. Since $\overline{\cup\mathcal{V}}$ is the smallest closed set in $X$ containing $\cup\mathcal{V}$, $\overline{\cup\mathcal{V}}\subseteq \cup\{\overline{V} : V\in\mathcal{V}\}$ and hence $\cup\{\overline{V} : V\in\mathcal{V}\}=X$. Thus $X$ is almost Alster.
\end{proof}

Thus we obtain the following.
\begin{Th}
\label{T44}
For a $P$-space $X$ the following properties are equivalent.
\begin{multicols}{2}
\begin{enumerate}[wide=0pt,label={\upshape(\arabic*)},leftmargin=*]
  \item $\aL$
  \item $\aM$
  \item $\aS$
  \item almost Alster
  \item $\wL$
  \item $\wM$
  \item $\wS$
  \item weakly Alster.
\end{enumerate}
\end{multicols}
\end{Th}

\begin{Cor}
\label{C28}
Every $\wSk$ $P$-space is almost Alster (and hence $\aS$).
\end{Cor}

Combining \cite[Corollary 2.5]{MTCP} and Theorem~\ref{T26} we have the following.
\begin{Cor}
\label{C29}
For a regular $P$-space $X$ the following properties are equivalent.
\begin{multicols}{3}
\begin{enumerate}[wide=0pt,label={\upshape(\arabic*)},leftmargin=*]
  \item Lindel\"{o}f
  \item Menger
  \item Scheepers
  \item $\aL$
  \item $\aM$
  \item $\aS$
  \item almost Alster
  \item $\wL$
  \item $\wM$
  \item $\wS$
  \item weakly Alster
  \item $\wSk$.
\end{enumerate}
\end{multicols}
\end{Cor}

\subsection{Results concerning cardinalities}
For a set $Y\subseteq\mathbb{N}^\mathbb{N}$, $\maxfin(Y)$ is defined as $\maxfin(Y)=\{\max\{f_1,f_2,\ldots,f_k\} : f_1,f_2,\ldots,f_k\in Y\;\text{and}\;k\in\mathbb{N}\}$, where $\max\{f_1,f_2,\ldots,f_k\}\in\mathbb{N}^\mathbb{N}$ is given by \[\max\{f_1,f_2,\ldots,f_k\}(n)=\max\{f_1(n),f_2(n),
\ldots,f_k(n)\}\] for all $n\in\mathbb{N}$.

\begin{Th}
\label{T18}
Every $\aL$ space $X$ with cardinality less than $\mathfrak{d}$ is $\aS$.
\end{Th}
\begin{proof}
 Let $(\mathcal{U}_n)$ be a sequence of open covers of $X$. For each $n$ choose  a countable set $\mathcal{W}_n=\{V_m^{(n)} : m\in\mathbb{N}\}\subseteq\mathcal{U}_n$ such that $\cup_{V\in\mathcal{W}_n}\overline{V}=X$. Now for each $x\in X$  define a $f_x\in\mathbb{N}^\mathbb{N}$ by $f_x(n)=\min\{m\in\mathbb{N} : x\in\overline{V_m^{(n)}}\}$, $n\in\mathbb{N}$. Since the cardinality of $Y=\{f_x : x\in X\}$ is less than $\mathfrak{d}$, $\maxfin(Y)$ is also of cardinality less than $\mathfrak{d}$. Consequently there are a $g\in\mathbb{N}^\mathbb{N}$ and a $n_F\in\mathbb{N}$ corresponding to each finite set $F\subseteq X$ such that $f_F(n_F)<g(n_F)$ with $f_F\in\maxfin(Y)$, where $f_F(n)=\max\{f_x(n) : x\in F\}$ for all $n\in\mathbb{N}$. We use the convention that if $F=\{x\}$, $x\in X$, then we write $f_x$ instead of $f_F$. Observe that for each $n$ $\mathcal{V}_n=\{V_m^{(n)} : m\leq g(n)\}$ is a finite subset of $\mathcal{U}_n$. We claim that the sequence $(\mathcal{V}_n)$ witnesses for $X$ to be $\aS$. To see this, let $F$ be a finite subset of $X$. Now choose a $n_F\in\mathbb{N}$ and a $f_F\in\maxfin(Y)$ with $f_F(n)=\max\{f_x(n) : x\in F\}$ for all $n\in\mathbb{N}$ such that $f_F(n_F)<g(n_F)$. From the construction of $Y$, it follows that each $x\in F$ belongs to $\overline{V_{f_x(n)}^{(n)}}$ for all $n\in\mathbb{N}$. Consequently  $f_x(n_F)<g(n_F)$ for each $x\in F$ and hence $F\subseteq\overline{\cup\mathcal{V}_{n_F}}$. This completes the proof.
\end{proof}

\begin{Th}
\label{T20}
Let $X$ be Lindel\"{o}f and $\kappa<\mathfrak d$. If $X$ is a union of $\kappa$ many $\aH$ spaces, then $X$ is $\aS$.
\end{Th}

\begin{proof}
Let  $X=\cup_{\alpha<\kappa} X_\alpha$ where each $X_\alpha$ is $\aH$ and $\kappa<\mathfrak d$. Choose a sequence $(\mathcal{U}_n)$ of open covers of $X$ and without loss of generality assume that $\mathcal{U}_n=\{U_m^{(n)} : m\in\mathbb{N}\}$ for each $n$. For each $\alpha<\kappa$ choose a sequence $(\mathcal{V}_n^{(\alpha)})$ such that for each $n$ $\mathcal{V}_n^{(\alpha)}$ is a finite subset of $\mathcal{U}_n$ and each $x\in X_\alpha$ belongs to $\overline{\cup\mathcal{V}_n^{(\alpha)}}$ for all but finitely many $n$. Also for each $\alpha<\kappa$  define $f_\alpha:\mathbb{N}\to\mathbb{N}$ by $f_\alpha(n)=\min\{m\in\mathbb{N} : \mathcal{V}_n^{(\alpha)}\subseteq\{U_i^{(n)} : i\leq m\}\}$. If $Y=\{f_\alpha : \alpha<\kappa\}$, then the cardinality of $\maxfin(Y)$ is less than $\mathfrak{d}$. Thus there is a $g\in\mathbb{N}^\mathbb{N}$ such that for each finite subset $A$ of $\kappa$ we have $g\nleq^* f_A$ where $f_A\in\maxfin(Y)$. Clearly  $\mathcal{V}_n=\{U_i^{(n)} : i\leq g(n)\}$ is a finite subset of $\mathcal{U}_n$ for each $n$. We now show that the sequence $(\mathcal{V}_n)$ witnesses that $X$ is $\aS$. Let $F$ be a finite subset of $X$. Choose a finite subset $A$ of $\kappa$ such that $F=\cup_{\alpha\in A}F_\alpha$ with $F_\alpha\subseteq X_\alpha$. For each $\alpha\in A$, choose a $n_\alpha\in\mathbb{N}$ such that $F_\alpha\subseteq \overline{\cup\mathcal{V}_n^{(\alpha)}}$ for all $n\geq n_\alpha$. If we choose $n_0=\max\{n_\alpha : \alpha\in A\}$, then we can find a $n_1\in\mathbb{N}$ such that $n_1>n_0$ and $f_A(n_1)<g(n_1)$. Thus for each $\alpha\in A$, we have $F_\alpha\subseteq\cup_{i\leq f_\alpha(n_1)}\overline{U_i^{(n_1)}}\subseteq\cup_{i\leq f_A(n_1)}\overline{U_i^{(n_1)}}$, i.e. $F\subseteq \overline{\cup\mathcal{V}_{n_1}}$. The conclusion now follows.
\end{proof}

A space $X$ is said to be $H$-closed if each open cover of $X$ contains a finite subfamily whose union is dense in $X$ \cite{Engelking}. Since every $H$-closed space is $\aH$, we have the following.
\begin{Cor}
\label{C16}
Let $X$ be Lindel\"{o}f and $\kappa<\mathfrak d$. If $X$ is a union of $\kappa$ many $H$-closed subspaces, then $X$ is $\aS$.
\end{Cor}

Also since $\aS$ implies $\aM$, we obtain the following result of Di Maio and Ko\v{c}inac as another consequence of the preceding theorem.
\begin{Cor}[\!{\cite[Theorem 2.5]{QM}}]
\label{C17}
Let $X$ be Lindel\"{o}f and $\kappa<\mathfrak d$. If $X$ is a union of $\kappa$ many $H$-closed subspaces, then $X$ is $\aM$.
\end{Cor}
The following result is similar to Theorem~\ref{T20} with necessary modifications.
\begin{Th}
Let $X$ be Lindel\"{o}f and $\kappa<\mathfrak d$.
\begin{enumerate}[wide=0pt,label={\upshape(\arabic*)}]
\item\label{T21} If $X$ is a union of $\kappa$ many $\wH$ spaces, then $X$ is $\wS$.
\item\label{T22} If $X$ is a union of $\kappa$ many $\wHk$ spaces, then $X$ is $\wSk$.
\end{enumerate}
\end{Th}

Let $\kappa$ be any cardinal. We say that the collection $\{X_\alpha : \alpha<\kappa\}$ is a $\Omega$-wrapping if for each finite set $F\subseteq\cup_{\alpha<\kappa}X_\alpha$ there exists a $\beta<\kappa$ such that $F\subseteq X_\beta$.
\begin{Th}
Let $X$ be Lindel\"{o}f and $\kappa<\mathfrak b$. Suppose that $X=\cup_{\alpha<\kappa}X_\alpha$ and the collection $\{X_\alpha : \alpha<\kappa\}$ is a $\Omega$-wrapping. The following assertions hold.
\begin{enumerate}[wide=0pt,label={\upshape(\arabic*)}]
\item \label{T23} If each $X_\alpha$ is $\aS$, then $X$ is $\aS$.
\item \label{T24} If each $X_\alpha$ is $\wS$, then $X$ is $\wS$.
\end{enumerate}
\end{Th}

\begin{proof}
We only give proof of (1).
Let $(\mathcal{U}_n)$ be a sequence of open covers of $X$. Without loss of generality assume that for each $n$ $\mathcal{U}_n=\{U_m^{(n)} : m\in\mathbb{N}\}$. For each $\alpha<\kappa$ choose a sequence $(\mathcal{V}_n^{(\alpha)})$ such that each $\mathcal{V}_n^{(\alpha)}$ is a finite subset of $\mathcal{U}_n$ and each finite set $F\subseteq X_\alpha$ is contained in $\overline{\cup\mathcal{V}_n^{(\alpha)}}$ for infinitely many $n$. Also for each $\alpha<\kappa$ define $f_\alpha:\mathbb{N}\to\mathbb{N}$ by $f_\alpha(n)=\min\{m\in\mathbb{N} : \mathcal{V}_n^{(\alpha)}\subseteq\{U_i^{(n)} : i\leq m\}\}$. Since $\kappa<\mathfrak{b}$, there is a $g\in\mathbb{N}^\mathbb{N}$ such that for each $\alpha<\kappa$ we have $f_\alpha\leq^* g$. Observe that $\mathcal{V}_n=\{U_i^{(n)} : i\leq g(n)\}$ is a finite subset of $\mathcal{U}_n$ for each $n$. We now claim that the sequence $(\mathcal{V}_n)$ witnesses that $X$ is $\aS$. Let $F$ be a finite subset of $X$. Choose a $\beta<\kappa$ such that $F\subseteq X_\beta$. Thus we have $F\subseteq\overline{\cup\mathcal{V}_n^{(\beta)}}$ for infinitely many $n$. Also choose a $n_0\in\mathbb{N}$ such that $f_\beta(n)\leq g(n)$ for all $n\geq n_0$. Again choose a $n_1>n_0$ such that $F\subseteq\overline{\cup\mathcal{V}_{n_1}^{(\beta)}}$. Clearly $F\subseteq\cup_{i\leq f_\beta(n_1)}\overline{U_i^{(n_1)}}\subseteq\cup_{i\leq g(n_1)}\overline{U_i^{(n_1)}}$ and hence $F\subseteq\overline{\cup\mathcal{V}_{n_1}}$. The claim is now proved.
\end{proof}

\begin{Rem}
\label{R3}
We do not know whether the above result holds for the $\wSk$ property.
\end{Rem}
A collection $\{X_\alpha : \alpha<\kappa\}$ in a space $X$ is said to be a strongly $\Omega$-wrapping if for each $\alpha<\kappa$ and any dense set $Y_\alpha\subseteq X_\alpha$ the collection $\{Y_\alpha : \alpha<\kappa\}$ is a $\Omega$-wrapping.
\begin{Th}
\label{T55}
Let $X$ be Lindel\"{o}f and $\kappa<\mathfrak b$. Suppose that $X=\cup_{\alpha<\kappa}X_\alpha$ and $\{X_\alpha : \alpha<\kappa\}$ is a strongly $\Omega$-wrapping. If each $X_\alpha$ is $\wSk$, then $X$ is $\wSk$.
\end{Th}
\begin{proof}
Let $(\mathcal{U}_n)$ be a sequence of open covers of $X$. We can assume that for each $n$ $\mathcal{U}_n=\{U_m^{(n)} : m\in\mathbb{N}\}$. For each $\alpha<\kappa$ we get a dense subset $Y_\alpha$ and a sequence $(\mathcal{V}_n^{(\alpha)})$ such that for each $n$ $\mathcal{V}_n^{(\alpha)}$ is a finite subset of $\mathcal{U}_n$ and each finite set $F\subseteq Y_\alpha$ is contained in $\cup\mathcal{V}_n^{(\alpha)}$ for infinitely many $n$. By the given condition, $\{Y_\alpha : \alpha<\kappa\}$ is a $\Omega$-wrapping. For each $\alpha<\kappa$ we define $f_\alpha:\mathbb{N}\to\mathbb{N}$ by $f_\alpha(n)=\min\{m\in\mathbb{N} : \mathcal{V}_n^{(\alpha)}\subseteq\{U_i^{(n)} : i\leq m\}\}$. Since $\kappa<\mathfrak{b}$, there exists a $g\in\mathbb{N}^\mathbb{N}$ such that for each $\alpha<\kappa$ we get $f_\alpha\leq^* g$. Clearly $Y=\cup_{\alpha<\kappa}Y_\alpha$ is a dense subset of $X$ and for each $n$ $\mathcal{V}_n=\{U_i^{(n)} : i\leq g(n)\}$ is a finite subset of $\mathcal{U}_n$. It is easy to observe that $Y$ and the sequence $(\mathcal{V}_n)$ witness for $(\mathcal{U}_n)$ that $X$ is $\wSk$.
\end{proof}

\section{Further observations on the weaker versions}
\subsection{Preservation like properties}
\label{S1}
We now present few preservation like properties of these spaces. First observe that every regular-closed subset of a $\wSk$ space is $\wSk$.
The above result does not hold for the other variations.
\begin{Ex}\hfill
\begin{enumerate}[wide=0pt,label={\upshape(\arabic*)}]
\item \label{E6}
\emph{A regular-closed subset of a $\aS$ space need not be $\aS$.}\\
Let $X_1$ be the space as in Example~\ref{E1} and $X_2$ be the space as in Example~\ref{E2}. Then $X_1$ is $\aH$ and $X_2$ is not $\aS$. Assume that $X_1\cap X_2=\emptyset$. Since the cardinality of $D$ is $\omega_1$, we write $D=\{d_\alpha :\alpha<\omega_1\}$. Define a bijection $\varphi:D\times\{\omega\}\to A$ by $\varphi(d_\alpha,\omega)=(a_\alpha,-1)$ for each $\alpha<\omega_1$. Also define $Z$ to be the quotient image of the topological sum $X_1\oplus X_2$ by identifying $(d_\alpha,\omega)$ of $X_2$ with $\varphi(d_\alpha,\omega)$ of $X_1$ for each $\alpha<\omega_1$. Let $q:X_1\oplus X_2\to Z$ be the quotient map. Now $q(X_2)$ is a regular-closed subset of $Z$ which is not $\aS$ as it is homeomorphic to $X_2$.

We now claim that $Z$ is $\aS$. The claim will follow if we show that $Z$ is $\aH$. Choose a sequence  $(\mathcal{U}_n)$ of open covers of $Z$. Now $q(X_1)$ being the homoeomorphic image of a $\aH$ space, is also $\aH$. Apply the $\aH$ property of $q(X_1)$ to $(\mathcal{U}_n)$ to obtain a sequence $(\mathcal{H}_n)$ such that for each $n$ $\mathcal{H}_n$ is a finite subset of $\mathcal{U}_n$ and each $x\in q(X_1)$ belongs to $\overline{\cup\mathcal{H}_n}$ for all but finitely many $n$. Again since $q(\beta D\times\omega)$ is homeomorphic to $\beta D\times\omega$, $q(\beta D\times\omega)$ is $\sigma$-compact and so is Hurewicz. Thus there is a sequence $(\mathcal{K}_n)$ such that for each $n$ $\mathcal{K}_n$ is a finite subset of $\mathcal{U}_n$ and each $x\in q(\beta D\times\omega)$ belongs to $\cup\mathcal{K}_n$ for all but finitely many $n$. For each $n$ let $\mathcal{V}_n=\mathcal{H}_n\cup\mathcal{K}_n$. The sequence $(\mathcal{V}_n)$ witnesses that $Z$ is $\aH$.\\
\item
\label{E7}
\emph{A closed subset of a $\wS$ (respectively, $\wSk$) space need not be $\wS$ (respectively, $\wSk$).}\\
Consider $X$ as in Example~\ref{E2}. Now $X$ is a Tychonoff $\wSk$ space and hence a $\wS$ space. Since $D\times\{\omega\}$ is a discrete closed subset of $X$ with cardinality $\omega_1$, it follows that $D\times\{\omega\}$ fails to be $\wS$ and $\wSk$ as well.
\end{enumerate}
\end{Ex}

Observe that if a subset $Y$ of a space $X$ is $\wS$ (respectively, $\wSk$), then $\overline{Y}$ is also $\wS$ (respectively, $\wSk$). Thus for a dense subset $Y$ of $X$, if $Y$ is $\wS$ (respectively, $\wSk$), then $X$ is also $\wS$ (respectively, $\wSk$), and also every separable space is $\wSk$ (and hence $\wS$). Surprisingly a $\aS$ space may not satisfy this preservation property. The Baire space $X$ is not $\aS$ because $X$ is paracompact and is not Scheepers (see Theorem~\ref{T26}). Since $X$ is separable, there exists a countable dense subset $Y$ of $X$. Thus $Y$ is $\aS$ but $\overline{Y}=X$ is not so.

\begin{Rem}
\label{R1}
The Baire space $X$ is hypocompact and separable. Clearly $X$ is $\wSk$ (and hence $\wS$). Thus there is a hypocompact $\wSk$ space which is not Scheepers (as $X$ is not $\aS$).
\end{Rem}

%
%
%
%
%

\begin{Th}
\label{T14}
For a space $X$ the following assertions are equivalent.
\begin{enumerate}[wide=0pt,label={\upshape(\arabic*)}]
  \item $X$ is Scheepers.
  \item $AD(X)$ is Scheepers.
  \item $AD(X)$ is $\wSk$.
\end{enumerate}
\end{Th}
\begin{proof}
$(1)\Rightarrow (2)$. Let $(\mathcal{U}_n)$ be a sequence of open covers of $AD(X)$. For each $n$ and each $x\in X$ let $W_x^{(n)}=(V_x^{(n)}\times\{0,1\})\setminus\{(x,1)\}$ be an open set in $AD(X)$ containing $(x,0)$ such that there is a $U_x^{(n)}\in\mathcal{U}_n$ with $W_x^{(n)}\subseteq U_x^{(n)}$, where $V_x^{(n)}$ is an open set in $X$ containing $x$. For each $n$ $\mathcal{W}_n=\{V_x^{(n)} : x\in X\}$ is an open cover of $X$. Apply (1) to $(\mathcal{W}_n)$ to obtain a sequence $(F_n)$ of finite subsets of $X$ such that $(\{V_x^{(n)} : x\in F_n\})$ witnesses the Scheepers property of $X$. For each $n$ and each $x\in F_n$ choose a $O_x^{(n)}\in\mathcal{U}_n$ with $(x,1)\in O_x^{(n)}$. Observe that  $\mathcal{V}_n=\{U_x^{(n)} : x\in F_n\}\cup\{O_x^{(n)} : x\in F_n\}$ is a finite subset of $\mathcal{U}_n$ for each $n$. The sequence $(\mathcal{V}_n)$ witnesses that $AD(X)$ is Scheepers.

$(3)\Rightarrow (1)$. Let $(\mathcal{U}_n)$ be a sequence of open covers of $X$. Say, $\mathcal{U}_n=\{U_x^{(n)} : x\in X\}$, where $U_x^{(n)}$ is an open set in $X$ containing $x$ for each $n$. Choose $\mathcal{W}_n=\{(U_x^{(n)}\times\{0,1\})\setminus\{(x,1)\} : x\in X\}\cup\{\{(x,1)\} : x\in X\}$ for each $n$. Since $(\mathcal{W}_n)$ is a sequence of open covers of $AD(X)$, there are a dense subset $Z$ of $AD(X)$ and a sequence $(F_n)$ of finite subsets of $X$ such that $\{(U_x^{(n)}\times\{0,1\})\setminus\{(x,1)\} : x\in F_n\}\cup\{\{(x,1)\} : x\in F_n\}$ witnesses the $\wSk$ property of $AD(X)$. For each $n$ $\mathcal{V}_n=\{U_x^{(n)} : x\in F_n\}$ is a finite subset of $\mathcal{U}_n$. It now follows that $X$ is Scheepers.
\end{proof}

We now give an example of a $\aS$ space $X$ such that $AD(X)$ is not $\aS$.
Let $X$ be the space as in Example~\ref{E1}. Thus $X$ is an Urysohn $\aS$ space and $A=\{(a_\alpha,-1) : \alpha<\omega_1\}$ is an uncountable discrete closed subset of $X$. Clearly $A\times\{1\}$ is an uncountable discrete clopen subset of $AD(X)$. Since the $\aS$ property is hereditary for clopen subsets, it follows that $AD(X)$ is not $\aS$.

Also there is a $\wS$ (respectively, $\wSk$) space $X$ such that $AD(X)$ is not $\wS$ (respectively, not $\wSk$).
Let $X$ be the space as in Example~\ref{E2}. Then $X$ is a Tychonoff $\wS$ (and also a $\wSk$) space and $A=D\times\{\omega\}$ is an uncountable discrete closed subset of $X$. Observe that $A\times\{1\}$ is an uncountable discrete clopen subset of $AD(X)$. Thus $AD(X)$ is not $\wS$. It is also clear that $AD(X)$ is not $\wSk$.

\begin{Th}
\label{T30}
For a space $X$
\begin{enumerate}[wide=0pt,label={\upshape(\arabic*)},
ref={\theTh(\arabic*)},leftmargin=*]
  \item\label{T3001} if $AD(X)$ is $\aS$, then $X$ is also $\aS$.
  \item\label{T3002} if $AD(X)$ is $\wS$, then $X$ is also $\wS$.
\end{enumerate}
\end{Th}
\begin{proof}
We only present proof for $(1)$.
Let $(\mathcal{U}_n)$ be a sequence of open covers of $X$. For each $n$ $\mathcal{W}_n=\{U\times\{0,1\} : U\in\mathcal{U}_n\}$ is an open cover of $AD(X)$. Apply the $\aS$ property of $AD(X)$ to $(\mathcal{W}_n)$ to obtain a sequence $(\mathcal{H}_n)$ such that for each $n$ $\mathcal{H}_n$ is a finite subset of $\mathcal{W}_n$ and each finite set $A\subseteq AD(X)$ is contained in $\overline{\cup\mathcal{H}_n}$ for some $n$. For each $n$ choose $\mathcal{V}_n=\{U\in\mathcal{U}_n : U\times\{0,1\}\in\mathcal{H}_n\}$. We claim that the sequence $(\mathcal{V}_n)$ witnesses for $(\mathcal{U}_n)$ that $X$ is $\aS$. Let $F$ be a finite subset of $X$. Choose a $n_0\in\mathbb{N}$ such that $F\times\{0\}\subseteq\overline{\cup\mathcal{H}_{n_0}}$. It follows that $F\subseteq \overline{\cup\mathcal{V}_{n_0}}$. This completes the proof.
\end{proof}

By Example~\ref{E1} (respectively, Example~\ref{E2}), there exists an Urysohn $\aS$ (respectively, Tychonoff $\wSk$ and hence $\wS$) space $X$ such that $e(X)\geq\omega_1$. However for a $T_1$ space $X$ it can be shown that if $AD(X)$ is $\wS$, then $e(X)<\omega_1$. The conclusion also holds if $AD(X)$ is either $\aS$ or $\wSk$.

\subsection{The productively properties}
If $P$ is a property of a space, we call a space $X$ productively $P$ if $X\times Y$ has the property $P$ whenever $Y$ has the property $P$. We start with the following basic observation for dense subsets. If $Y\subseteq X$ is dense in $X$ and $D\subseteq Y$ is dense in $Y$, then $D$ is dense in $X$.
Also if $D\subseteq X$ is dense in $X$ and $E\subseteq Y$ is dense in $Y$, then $D\times E$ is dense in $X\times Y$.


\begin{Prop}
\label{T35}
Let $D$ be a dense subset of a space $X$.
\begin{enumerate}[wide=0pt,label={\upshape(\arabic*)},leftmargin=*]
  \item If $D$ is productively $\wS$, then $X$ is also productively $\wS$.
  \item If $D$ is productively $\wSk$, then $X$ is also productively $\wSk$.
\end{enumerate}
\end{Prop}
\begin{proof}
We only give proof for the first assertion.
Let $Y$ be a $\wS$ space and $(\mathcal{U}_n)$ be a sequence of open covers of $X$. Since $D$ is productively $\wS$, $D\times Y$ is $\wS$. Choose a sequence $(\mathcal{V}_n)$ such that for each $n$ $\mathcal{V}_n$ is a finite subset of $\mathcal{U}_n$ and for each finite collection $\mathcal{F}$ of nonempty open sets of $D\times Y$ there is a $n$ such that $U\cap(\cup\mathcal{V}_n)\neq\emptyset$ for all $U\in\mathcal{F}$. Since $D$ is dense in $X$, the sequence $(\mathcal{V}_n)$ witnesses for $(\mathcal{U}_n)$ that $X\times Y$ is $\wS$.
\end{proof}

The above result does not hold for productively $\aS$ spaces (see Remark~\ref{R2}).

\begin{Prop}
\label{T36}
Let $\tau$ and $\tau^\prime$ be two topologies on $X$ such that $\tau^\prime$ is finer than $\tau$. The following assertions hold.
\begin{enumerate}[wide=0pt,label={\upshape(\arabic*)},leftmargin=*]
  \item If $(X,\tau^\prime)$ is productively $\aS$, then $(X,\tau)$ is so.
  \item If $(X,\tau^\prime)$ is productively $\wS$, then $(X,\tau)$ is so.
  \item If $(X,\tau^\prime)$ is productively $\wSk$, then $(X,\tau)$ is so.
\end{enumerate}
\end{Prop}

\begin{Th}
\label{T37}
A $H$-closed space $X$ is
\begin{enumerate}[wide=0pt,label={\upshape(\arabic*)},leftmargin=*]
  \item productively $\aS$.
  \item productively $\wS$.
  \item productively $\wSk$.
\end{enumerate}
\end{Th}
\begin{proof}

Let $Y$ be a $\aS$ space. Consider a sequence $(\mathcal{U}_n)$ of open covers of $X\times Y$. Without loss of generality assume that $\mathcal{U}_n=\mathcal{V}_n\times\mathcal{W}_n$ for each $n$, where $\mathcal{V}_n$ and $\mathcal{W}_n$ are respectively open covers of $X$ and $Y$. Fix $y\in Y$. Since $X\times\{y\}$ is $H$-closed, there is a sequence $(\mathcal{V}_n^y\times\mathcal{W}_n^y)$ such that for each $n$ $\mathcal{V}_n^y\times\mathcal{W}_n^y$ is a finite subset of $\mathcal{U}_n$ and $\cup(\mathcal{V}_n^y\times\mathcal{W}_n^y)$ is dense in $X\times\{y\}$. Consider the open cover $\mathcal{U}_n^\prime=\{U_n^y : y\in Y\}$ of $Y$, where for each $n$ $U_n^y=\cap\mathcal{W}_n^y$. Apply the $\aS$ property of $Y$ to $(\mathcal{U}_n^\prime)$ to obtain a sequence $(\mathcal{H}_n)$ such that for each $n$ $\mathcal{H}_n$ is a finite subset of $\mathcal{U}_n^\prime$ and for each finite set $F\subseteq Y$ there is a $n$ such that $F\subseteq\overline{\cup\mathcal{H}_n}$. For each $n$ choose $\mathcal{H}_n=\{U_n^{y_1},U_n^{y_2},\dotsc,U_n^{y_{k_n}}\}$. Now for each $n$ $\mathcal{K}_n=\cup_{i=1}^{k_n}(\mathcal{V}_n^{y_i}
\times\mathcal{W}_n^{y_i})$ is a finite subset of $\mathcal{U}_n$. Clearly the sequence $(\mathcal{K}_n)$ witnesses that  $X\times Y$ is $\aS$. Thus $X$ is productively $\aS$. The other two cases can be easily verified.
\end{proof}

\begin{Th}
\label{T38}
If $X$ satisfies $\S1(\mathcal{G}_K,\mathcal{G}_{\overline{\Gamma}})$, then $X$ is productively $\aS$.
\end{Th}
\begin{proof}
Let $Y$ be a space having the property $\aS$. To show that $X\times Y$ satisfies $\aS$ we choose a sequence $(\mathcal{U}_n)$ of open covers of $X\times Y$. Without loss of generality we assume that for each $n$ $\mathcal{U}_n$ is closed under finite unions. For each compact set $C\subseteq X$ and each $n$ choose a $G_\delta$ set $G_n(C)\subseteq X$ such that $C\subseteq G_n(C)$. Then for each $y\in Y$ and each $n$ there exists a $U\in\mathcal{U}_n$ such that $G_n(C)\times\{y\}\subseteq U$. Clearly $\mathcal{G}_n=\{G_n(C) : C\;\text{is a compact subset of}\; X\}\in\mathcal{G}_K$ for each $n$. Since $X$ satisfies $\S1(\mathcal{G}_K,\mathcal{G}_{\overline{\Gamma}})$, we obtain a sequence $(C_n)$ of compact subsets of $X$ such that each $x\in X$ belongs to $\overline{G_n(C_n)}$ for all but finitely many $n$. For each $n$ choose $\mathcal{W}_n=\{V : V\;\text{is an open set in $Y$ and there is a $U\in\mathcal{U}_n$ such that}\; G_n(C_n)\times V\subseteq U\}$. Clearly $(\mathcal{W}_n)$ is a sequence of open covers of $Y$. Since $Y$ satisfies $\aS$, there exists a sequence $(\mathcal{K}_n)$ such that for each $n$ $\mathcal{K}_n$ is a finite subset of $\mathcal{W}_n$ and each finite set $F\subseteq Y$ is contained in $\overline{\cup\mathcal{K}_n}$ for infinitely many $n$. For each $n$ and each $V\in\mathcal{K}_n$ choose a $U_V\in\mathcal{U}_n$ such that $G_n(C_n)\times V\subseteq U_V$. For each $n$ let $\mathcal{V}_n=\{U_V : V\in\mathcal{K}_n\}$. Observe that the sequence $(\mathcal{V}_n)$ fulfils the requirement for 
$X$ to be productively $\aS$.
\end{proof}
Similarly we obtain the following.
\begin{Th}
\label{T39}
If $X$ satisfies $\S1(\mathcal{G}_K,\mathcal{G}_{D_\Gamma})$, then $X$ is productively $\wS$.
\end{Th}

\begin{Th}
\label{T40}
If $X$ satisfies $\S1(\mathcal{G}_K,\mathcal{G}_{\Gamma_D})$, then $X$ is productively $\wSk$.
\end{Th}

The following result is an improvement of \cite[Lemma 44]{WCPSP}.
\begin{Th}
\label{T41}
If any dense subspace of $X$ satisfies $\S1(\mathcal{G}_K,\mathcal{G}_\Gamma)$, then $X$ also satisfies $\S1(\mathcal{G}_K,\mathcal{G}_{\Gamma_D})$.
\end{Th}
\begin{proof}
Let $Y$ be a dense subset of $X$ having the property $\S1(\mathcal{G}_K,\mathcal{G}_\Gamma)$. We pick a sequence $(\mathcal{U}_n)$ of members of $\mathcal{G}_K$ to show that $X$ has the property $\S1(\mathcal{G}_K,\mathcal{G}_{\Gamma_D})$. Since $Y$ satisfies $\S1(\mathcal{G}_K,\mathcal{G}_\Gamma)$, there exists a sequence $(U_n)$ such that for each $n$ $U_n\in\mathcal{U}_n$ and $\{U_n : n\in\mathbb{N}\}\in\mathcal{G}_\Gamma$ for $Y$. It follows that $\{U_n : n\in\mathbb{N}\}\in\mathcal{G}_{\Gamma_D}$ for $X$ as $Y$ is dense in $X$. Thus $X$ has the property $\S1(\mathcal{G}_K,\mathcal{G}_{\Gamma_D})$.
\end{proof}

Since every $\sigma$-compact space has the property $\S1(\mathcal{G}_K,\mathcal{G}_\Gamma)$ (see \cite[Theorem 22]{WCPSP}), we have the following.
\begin{Cor}
\label{C24}
If a space $X$ has a dense $\sigma$-compact subset, then $X$ satisfies $\S1(\mathcal{G}_K,\mathcal{G}_{\Gamma_D})$.
\end{Cor}

\begin{Cor}
\label{C25}
Every separable space satisfies $\S1(\mathcal{G}_K,\mathcal{G}_{\Gamma_D})$.
\end{Cor}

Using Theorem~\ref{T39}, Theorem~\ref{T40} (and also Figure~\ref{dig2}) we have the following.
\begin{Cor}
\label{C26}
A separable space $X$ is
\begin{enumerate}[wide=0pt,label={\upshape(\arabic*)},leftmargin=*]
  \item productively $\wS$.
  \item productively $\wSk$.
\end{enumerate}
\end{Cor}


\begin{Cor}
\label{C27}
For each cardinal $\kappa$, $\mathbb{R}^\kappa$ is
\begin{enumerate}[wide=0pt,label={\upshape(\arabic*)},leftmargin=*]
  \item productively $\wS$
  \item productively $\wSk$
\end{enumerate}
\end{Cor}
\begin{proof}
Since $\{f\in\mathbb{R}^\kappa : |\{i : f(i)\neq 0\}|<\omega\}$ is a dense $\sigma$-compact subset of $\mathbb{R}^\kappa$ (see \cite[Proposition 4]{NSPS}), by Theorem~\ref{T39}, Theorem~\ref{T40}, Corollary~\ref{C24} and Figure~\ref{dig2}, $\mathbb{R}^\kappa$ has the claimed properties.
\end{proof}

The following example shows that a separable space need not be productively $\aS$.
\begin{Ex}
\label{E8}
\emph{There exists a separable space $X$ and a $\aS$ space $Y$ such that $X\times Y$ is not $\aS$.}\\
Let $\Psi(\mathcal{A})$ be the Isbell-Mr\'{o}wka space \cite{Mrowka} with $|\mathcal{A}|=\mathfrak{c}$, say $\mathcal{A}=\{a_\alpha : \alpha<\mathfrak{c}\}$. Then $X=\Psi(\mathcal{A})$ is separable. Let $D=\{d_\alpha : \alpha<\mathfrak{c}\}$ be the discrete space of cardinality $\mathfrak{c}$ and $Y=D\cup\{d\}$ be the one point compactification of $D$. Observe that $Y$ is $\aS$.

 For each $\alpha<\mathfrak{c}$ define $U_\alpha=(\{a_\alpha\}\cup a_\alpha)\times(Y\setminus\{d_\alpha\})$ and $V_\alpha=X\times\{d_\alpha\}$. Clearly both $U_\alpha$ and $V_\alpha$ are closed. For each $n$ choose $W_n=\{n\}\times Y$. We now define for each $n$ $\mathcal{U}_n=\{U_\alpha : \alpha<\mathfrak{c}\}\cup\{V_\alpha : \alpha<\mathfrak{c}\}\cup\{W_n : n\in\mathbb{N}\}$. Clearly $(\mathcal{U}_n)$ is a sequence of open covers of $X\times Y$. If $X\times Y$ is $\aS$, then there is a sequence $(\mathcal{V}_n)$ such that for each $n$ $\mathcal{V}_n$ is a finite subset of $\mathcal{U}_n$ and for each finite set $F\subseteq X\times Y$ there is a $n$ such that $F\subseteq\overline{\cup\mathcal{V}_n}$. Since $\cup_{n\in\mathbb{N}}\mathcal{V}_n$ is countable, choose a $\beta<\mathfrak{c}$ such that $V_\beta\notin\cup_{n\in\mathbb{N}}\mathcal{V}_n$. It now follows that $\{(a_\beta,d_\beta)\}\nsubseteq
\cup\mathcal{V}_n=\overline{\cup\mathcal{V}_n}$ for all $n$, as $V_\beta$ is the only member of $\mathcal{U}_n$ containing $(a_\beta,d_\beta)$ for each $n$. Which is a contradiction. Hence $X\times Y$ is not $\aS$.
\end{Ex}

By \cite[Theorem 3.9]{coc2} and \cite[Theorem 2.5]{FPM}, there exist two sets of reals $X$ and $Y$ such that both of them are Scheepers but their product $X\times Y$ is not Menger (hence not $\aS$, see Theorem~\ref{T26}). By Corollary~\ref{C26}, we can say that the property $\wS$ ($\wSk$) is preserved under finite products (in the case of sets of reals). But for arbitrary spaces $\wS$ property is not preserved under finite products (see Example~\ref{E10}). We need the following observations on the Pixley-Roy spaces (from \cite{WCPSP,PRS}).

\begin{Lemma}[\!{\cite[Proposition 4]{WCPSP}}]
\label{L1}
For any two spaces $X$ and $Y$, $\PR(X)\times \PR(Y)$ is homeomorphic to $\PR(X\oplus Y)$.
\end{Lemma}

\begin{Lemma}[\!{\cite[Proposition 5]{WCPSP}}]
\label{L3}
Assume CH. There exist separable metrizable spaces $X$ and $Y$ such that both have the property $\Sf(\Omega,\Omega)$ but $X\oplus Y$ does not have the Menger property.
\end{Lemma}

\begin{Th}[\!{\cite[Theorem 2A]{PRS}}]
\label{T34}
If $\PR(X)$ is $\wM$, then every finite power of $X$ is Menger.
\end{Th}

\begin{Th}[\!{\cite[Theorem 2B]{PRS}}]
\label{T33}
If $X$ is a metrizable space such that every finite power of $X$ is Menger, then $\PR(X)^\kappa$ is $\wM$ for each cardinal $\kappa$.
\end{Th}

\begin{Ex}
\label{E10}
Assume CH. There are $\wS$ spaces $X$ and $Y$ such that $X\times Y$ is not $\wM$.\\
By Lemma~\ref{L3}, there are two separable metrizable spaces $X$ and $Y$ such that both $X$ and $Y$ satisfy $\Sf(\Omega,\Omega)$ but $X\oplus Y$ does not satisfy the Menger property. Now every finite power of $X$ and also of $Y$ is Menger (see \cite[Theorem 3.9]{coc2}). By Theorem~\ref{T33}, every finite power of $\PR(X)$ and also of $\PR(Y)$ is $\wM$. It follows that both $\PR(X)$ and $\PR(Y)$ are $\wS$ (see Theorem~\ref{T5}). Also by Theorem~\ref{T34}, $\PR(X\oplus Y)$ is not $\wM$ and $\PR(X)\times\PR(Y)$ is not so (see Lemma~\ref{L1}).
\end{Ex}

\begin{figure}[h]
\begin{adjustbox}{max width=\textwidth,max height=\textheight,keepaspectratio,center}
\begin{tikzcd}[column sep=2ex,row sep=5ex,arrows={crossing over}]
&&&\wHk\arrow[rd]&
\\
\text{\sf H}\arrow[rr]\arrow[rrru]&&
\aH\arrow[rr]&&
\wH
\\
&&&\S1(\mathcal{G}_K,\mathcal{G}_{\Gamma_D})\arrow[rd]\arrow[uu]&
\\
\S1(\mathcal{G}_K,\mathcal{G}_\Gamma)\arrow[rr]\arrow[uu]\arrow[rrru]\arrow[d]&&
\S1(\mathcal{G}_K,\mathcal{G}_{\overline{\Gamma}})\arrow[rr]\arrow[uu]\arrow[d]&&
\S1(\mathcal{G}_K,\mathcal{G}_{D_\Gamma})\arrow[uu]\arrow[d]
\\
\S1(\mathcal{G}_K,\mathcal{G})\arrow[rr]\arrow[d]&&
\S1(\mathcal{G}_K,\overline{\mathcal{G}})\arrow[rr]\arrow[d]&&
\S1(\mathcal{G}_K,\mathcal{G}_D)\arrow[d]
\\
\text{\sf S}\arrow[rr]&&
\aS\arrow[rr]&&
\wS
\end{tikzcd}
\end{adjustbox}
\caption{Weaker forms of the Alster, Hurewicz and Scheepers properties}
\label{dig2}
\end{figure}
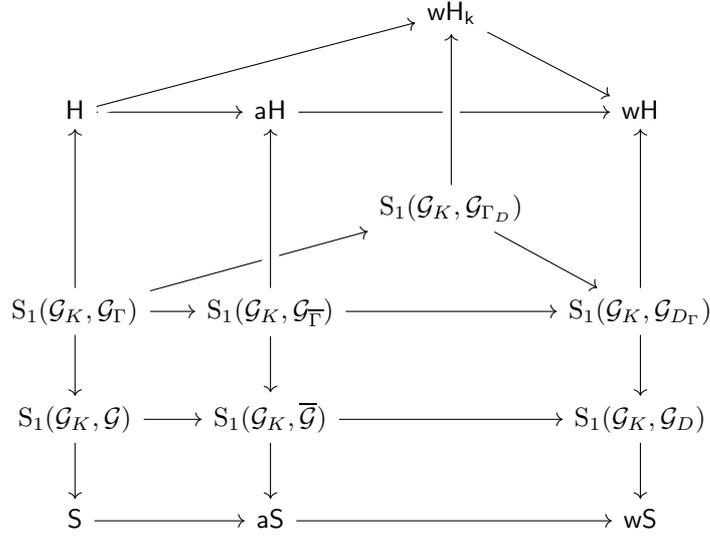

\subsection{Preservation under certain mappings}
We begin this section by recalling the following definitions of certain mappings which will come to our use. Let $X$ and $Y$ be two spaces. A mapping $f:X\to Y$ is called
\begin{enumerate}[wide=0pt,leftmargin=*]
  \item $\theta$-continuous \cite{ETS1} (respectively, strongly $\theta$-continuous \cite{STF}) if for each $x\in X$, and each open set $V$ in $Y$ containing $f(x)$ there is an open set $U$ in $X$ containing $x$ such that $f(\overline{U})\subseteq\overline{V}$ (respectively, $f(\overline{U})\subseteq V$).
  \item contra-continuous \cite{CCF} if for each open set $V$ in $Y$, $f^{-1}(V)$ is closed in $X$.
  \item precontinuous \cite{PWPM} if $f^{-1}(V)\subseteq\Int(\overline{f^{-1}(V)})$ whenever $V$ is open in $Y$.
  \item almost continuous \cite{ACM} if for each regular-open set $V$ in $Y$, $f^{-1}(V)$ is open in $X$.
  \item weakly continuous \cite{DCTS} if for each $x\in X$ and each open set $V$ of $Y$ containing $f(x)$ there is an open set $U$ of $X$ containing $x$ such that $f(U)\subseteq\overline{V}$.
  \item $\theta$-open \cite{Wilansky} if for each open set $V$ of $Y$, $f^{-1}(\overline{V})\subseteq\overline{f^{-1}(V)}$.
  \item almost open \cite{Wilansky} if for each $x\in X$ and each open set $U$ of $X$ containing $x$, $\overline{f(U)}$ is a neighbourhood of $f(x)$.
\end{enumerate}

It is immediate that every open map is $\theta$-open and every $\theta$-open map is almost open. Also one can verify the following implication diagram (Figure~\ref{dig3}).
\begin{figure}[h]
\begin{adjustbox}{max width=\textwidth,max height=\textheight,keepaspectratio,center}
\begin{tikzcd}[column sep=1.5ex,row sep=4ex,arrows={crossing over}]
\text{\small almost continuous}\arrow[r] & \text{\small $\theta$-continuous}\arrow[r] & \text{\small weakly continuous}
\\
\text{\small continuous}\arrow[r]\arrow[u] & \text{\small strongly $\theta$-continuous}\arrow[u] &
\\
\text{\small contra-continuous} + \text{\small precontinuous}\arrow[u] &&
\end{tikzcd}
\end{adjustbox}
\caption{Weaker forms of continuous mapping}
\label{dig3}
\end{figure}
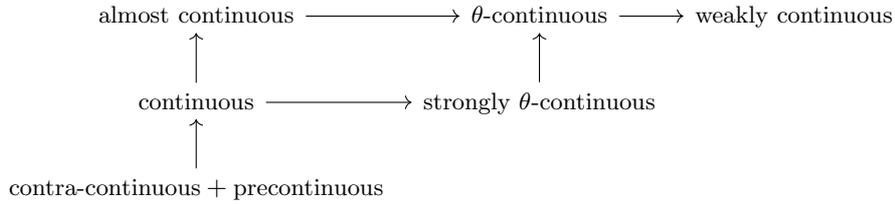

\begin{Th}
\label{T7}
If $f:X\to Y$ is a $\theta$-continuous mapping from a $\aS$ space $X$ onto a space $Y$, then $Y$ is also $\aS$.
\end{Th}
\begin{proof}
Let $(\mathcal{U}_n)$ be a sequence of open covers of $Y$. Let $x\in X$. For each $n$ choose a $V_x^{(n)}\in\mathcal{U}_n$ containing $f(x)$. Since $f$ is $\theta$-continuous, there is an open set $U_x^{(n)}$ in $X$ containing $x$ such that $f\left(\overline{U_x^{(n)}}\right)
\subseteq\overline{V_x^{(n)}}$. For each $n$ consider the open cover $\mathcal{W}_n=\{U_x^{(n)} : x\in X\}$ of $X$. Using the $\aS$ property of $X$ we choose a sequence $(\mathcal{H}_n)$ such that for each $n$ $\mathcal{H}_n=\{U_{x_i}^{(n)} : 1\leq i\leq k_n\}$ is a finite subset of $\mathcal{W}_n$ and for each finite set $F\subseteq X$ there is a $n$ such that $F\subseteq\overline{\cup\mathcal{H}_n}$. For each $n$ we consider the finite subset $\mathcal{V}_n=\{V_{x_i}^{(n)} : 1\leq i\leq k_n\}$ of $\mathcal{U}_n$. Let $F$ be a finite subset of $Y$. Next we pick a finite set $F^\prime\subseteq X$ such that $f(F^\prime)=F$. Now choose a $n_0\in\mathbb{N}$ such that $F^\prime\subseteq\overline{\cup\mathcal{H}_{n_0}}$. Since $f\left(\overline{U_{x_i}^{(n_0)}}\right)
\subseteq\overline{V_{x_i}^{(n_0)}}$ for each $1\leq i\leq k_{n_0}$, we have $F\subseteq\overline{\cup\mathcal{V}_{n_0}}$. Consequently $Y$ is $\aS$.
\end{proof}

\begin{Cor}
\label{C4}
The almost continuous image of a $\aS$ space is $\aS$.
\end{Cor}

Similarly we obtain the following.
\begin{Th}
\label{T10}
If $f:X\to Y$ is a weakly continuous mapping from a Scheepers space $X$ onto a space $Y$, then $Y$ is $\aS$.
\end{Th}

\begin{Th}
\label{T15}
Let $f:X\to Y$ be a surjective strongly $\theta$-continuous mapping.
\begin{enumerate}[wide=0pt,label={\upshape(\arabic*)}]
\item If $X$ is $\aS$, then $Y$ is Scheepers.
\item If $X$ is $\wS$, then $Y$ is also $\wS$.
\item If $X$ is $\wSk$, then $Y$ is also $\wSk$.
\end{enumerate}
\end{Th}

\begin{proof}
We only sketch the proof of (3). Let $(\mathcal{U}_n)$ be a sequence of open covers of $Y$. Let $x\in X$. For each $n$ choose a $V_x^{(n)}\in\mathcal{U}_n$ containing $f(x)$. Subsequently for each $n$ we obtain an open set $U_x^{(n)}$ in $X$ containing $x$ such that $f\left(\overline{U_x^{(n)}}\right)\subseteq V_x^{(n)}$. We now apply the $\wSk$ property of $X$ to the sequence $(\mathcal{W}_n)$, where for each $n$ $\mathcal{W}_n=\{U_x^{(n)} : x\in X\}$. Choose a dense set $Z\subseteq X$ and a sequence $(\mathcal{H}_n)$ such that for each $n$ $\mathcal{H}_n=\{U_{x_i}^{(n)} : 1\leq i\leq k_n\}$ is a finite subset of $\mathcal{W}_n$ and for each finite set $F\subseteq Z$ there is a $n$ such that $F\subseteq\cup\mathcal{H}_n$. Clearly $f(Z)$ is dense in $Y$. Let $\mathcal{V}_n=\{V_{x_i}^{(n)} : 1\leq i\leq k_n\}$ for each $n$. It is now easy to verify that $f(Z)$ and $(\mathcal{V}_n)$ witness that $Y$ is $\wSk$.
\end{proof}

\begin{Cor}
\label{C12}
Continuous image of a $\aS$ (respectively, $\wS$, $\wSk$) space is Scheepers (respectively, $\wS$, $\wSk$). The same is also true for contra-continuous precontinuous mappings.
\end{Cor}

\begin{Th}
\label{T11}
Let $f:X\to Y$ be a surjective, $\theta$-open and closed mapping such that $f^{-1}(y)$ is compact in $X$ for each $y\in Y$. If $Y$ is $\aS$, then $X$ is also $\aS$.
\end{Th}
\begin{proof}
Let $(\mathcal{U}_n)$ be a sequence of open covers of $X$ and $y\in Y$. Since $f^{-1}(y)$ is compact, for each $n$ there is a finite subset $\mathcal{V}_n^y$ of $\mathcal{U}_n$ such that $f^{-1}(y)\subseteq\cup\mathcal{V}_n^y$. Also there is an open set $U_y^{(n)}$ in $Y$ containing $y$ such that $f^{-1}(U_y^{(n)})\subseteq\cup\mathcal{V}_n^y$. For each $n$ $\mathcal{W}_n=\{U_y^{(n)} : y\in Y\}$ is an open cover of $Y$. Apply the $\aS$ property of $Y$ to $(\mathcal{W}_n)$ to obtain a sequence $(\mathcal{H}_n)$ such that for each $n$ $\mathcal{H}_n=\{U_{y_i}^{(n)} : 1\leq i\leq k_n\}$ is a finite subset of $\mathcal{W}_n$ and for each finite set $F\subseteq Y$ there is a $n$ such that $F\subseteq\overline{\cup\mathcal{H}_n}$. For each $n$ choose $\mathcal{V}_n=\cup_{1\leq i\leq k_n}\mathcal{V}_n^{y_i}$. Let $F$ be a finite subset of $X$. Then there is a $n_0\in\mathbb{N}$ such that $f(F)\subseteq\overline{\cup\mathcal{H}_{n_0}}$ i.e. $F\subseteq\cup_{1\leq i\leq k_{n_0}}f^{-1}\left(\overline{U_{y_i}^{(n_0)}}\right)$. Since $f$ is $\theta$-open, it follows that $F\subseteq\cup_{1\leq i\leq k_{n_0}}\overline{f^{-1}(U_{y_i}^{(n_0)})}$ and hence $F\subseteq\cup_{1\leq i\leq k_{n_0}}\overline{\cup\mathcal{V}_{n_0}^{y_i}}=
\overline{\cup\mathcal{V}_{n_0}}$. Thus $X$ is $\aS$.
\end{proof}

\begin{Cor}
\label{C8}
\hfill
\begin{enumerate}[wide=0pt,label={\upshape(\arabic*)},
leftmargin=*]
  \item If $f:X\to Y$ is a $\theta$-open (respectively, open), perfect mapping from a space $X$ onto a $\aS$ space $Y$, then $X$ is also $\aS$.
  \item  Let $f:X\to Y$ be a surjective, open and closed mapping such that $f^{-1}(y)$ is compact in $X$ for each $y\in Y$. If $Y$ is $\aS$, then $X$ is also $\aS$.
\end{enumerate}
\end{Cor}

\begin{Th}
\label{T12}
Let $f:X\to Y$ be a surjective, almost open and closed mapping such that $f^{-1}(y)$ is compact in $X$ for each $y\in Y$. If $Y$ is $\wS$, then $X$ is also $\wS$.
\end{Th}

\begin{Cor}
\label{C9}
\hfill
\begin{enumerate}[wide=0pt,label={\upshape(\arabic*)},
leftmargin=*]
  \item If $f:X\to Y$ is an almost open (respectively, $\theta$-open, open), perfect mapping from a space $X$ onto a $\wS$ space $Y$, then $X$ is also $\wS$.
  \item Let $f:X\to Y$ be a surjective, $\theta$-open (respectively, open) and closed mapping such that $f^{-1}(y)$ is compact in $X$ for each $y\in Y$. If $Y$ is $\wS$, then $X$ is also $\wS$.
\end{enumerate}
\end{Cor}

We also obtain a similar observation for the $\wSk$ property.
\begin{Th}
\label{T13}
Let $f:X\to Y$ be a surjective, open and closed mapping such that $f^{-1}(y)$ is compact in $X$ for each $y\in Y$. If $Y$ is $\wSk$, then $X$ is also $\wSk$.
\end{Th}


\begin{Cor}
\label{C10}
If $f:X\to Y$ is an open, perfect mapping from a space $X$ onto a $\wSk$ space $Y$, then $X$ is also $\wSk$.
\end{Cor}

A preimage of a $\aS$ (respectively, $\wS$, $\wSk$) space under a closed 2-to-1 continuous mapping may not be $\aS$ (respectively, $\wS$, $\wSk$). The space $X$ as in Example~\ref{E1} is $\aS$ but $AD(X)$ is not $\aS$. Also if we consider $X$ as in Example~\ref{E2}, then $X$ is a $\wSk$ (and hence $\wS$) space but $AD(X)$ is not $\wS$ (and hence not $\wSk$). In both cases it can be seen that the projection mapping $f:AD(X)\to X$ is closed 2-to-1 continuous.

\begin{Prop}
\label{P7}
Let $X=\cup_{m\in\mathbb{N}}X_m$, where $X_m\subseteq X_{m+1}$ for each $m$.
\begin{enumerate}[wide=0pt,label={\upshape(\arabic*)}]
\item If each $X_m$ is $\aS$, then $X$ is also $\aS$.
\item If each $X_m$ is $\wS$, then $X$ is also $\wS$.
\item If each $X_m$ is  $\wSk$, then $X$ is also $\wSk$.
    \end{enumerate}
\end{Prop}

In combination with Proposition~\ref{P7}, Theorem~\ref{T11}, Theorem~\ref{T12} and Theorem~\ref{T13} we obtain the following.
\begin{Prop}
\label{P5}
Let $X$ be a $\sigma$-compact space. Then $X$ is
\begin{enumerate}[wide=0pt,label={\upshape(\arabic*)}]
\item productively $\aS$.
\item productively $\wS$.
\item productively $\wSk$.
\end{enumerate}
\end{Prop}

\begin{Rem}
\label{R2}
Consider $X$ and $Y$ as in Example~\ref{E8}. Then $X$ is separable and $Y$ is $\aS$. Let $D$ be a countable dense subset of $X$. Since every $\sigma$-compact space is productively $\aS$ (by Proposition~\ref{P5}), $D$ is productively $\aS$. On the other hand, $X$ is not productively $\aS$ because $X\times Y$ is not $\aS$ (see Example~\ref{E8}).
\end{Rem}

\section{Few remarks}
The study of weaker forms of the Scheepers property may further be continued in the following ways. Throughout the section all covers are assumed to be countable.
\begin{Th}
\label{T28}
The following assertions are equivalent.
\begin{enumerate}[wide=0pt,label={\upshape(\arabic*)},leftmargin=*
,ref={\theTh(\arabic*)}]
  \item $X$ is $\aS$.
  \item $X$ satisfies $\Uf(\Gamma,\overline{\Omega})$.
  \item $X$ satisfies $\Sf(\Gamma,\overline{\Lambda^{wgp}})$.
  \item \label{T28-4} For each sequence $(\mathcal{U}_n)$ of $\gamma$-covers of $X$ there is a sequence $(\mathcal{V}_n)$ of pairwise disjoint finite sets such that for each $n$ $\mathcal{V}_n\subseteq\mathcal{U}_n$ and for each finite set $F\subseteq X$ there is a $n$ such that $F\subseteq\overline{\cup\mathcal{V}_n}$.
%

  \item $X$ satisfies $\Sf(\Gamma,\overline{\mathcal{O}^{wgp}})$.

  \item $X$ satisfies $\Uf(\mathcal{O},\overline{\mathcal{O}^{wgp}})$.
  \end{enumerate}
\end{Th}
\begin{proof}
$(2)\Rightarrow (1)$. Let $(\mathcal{U}_n)$ be a sequence of open covers of $X$. For each $n$ choose $\mathcal{U}_n=\{U_m^{(n)} : m\in\mathbb{N}\}$ and $\mathcal{W}_n=\{V_m^{(n)} : m\in\mathbb{N}\}$, where $V_m^{(n)}=\cup_{1\leq i\leq m}U_i^{(n)}$ for each $m$. Observe that $(\mathcal{W}_n)$ is a sequence of $\gamma$-covers of $X$. Since $X$ satisfies $\Uf(\Gamma,\overline{\Omega})$, there is a sequence $(V_{m_n}^{(n)})$ such that for each $n$ $V_{m_n}^{(n)}\in\mathcal{W}_n$ and $\{V_{m_n}^{(n)} : n\in\mathbb{N}\}\in\overline{\Omega}$. For each $n$ let $\mathcal{V}_n=\{U_i^{(n)} : 1\leq i\leq m_n\}$. Clearly $\cup\mathcal{V}_n=V_{m_n}^{(n)}$. It follows that $X$ satisfies $\Uf(\mathcal{O},\overline{\Omega})$.

$(2)\Rightarrow (3)$. Let $(\mathcal{U}_n)$ be a sequence of $\gamma$-covers of $X$. Without loss of generality we can assume that $\mathcal{U}_m\cap\mathcal{U}_n=\emptyset$ for $m\neq n$. Using the hypothesis we obtain a sequence $(\mathcal{V}_n)$ such that for each $n$ $\mathcal{V}_n$ is a finite subset of $\mathcal{U}_n$ and $\{\cup\mathcal{V}_n : n\in\mathbb{N}\}\in\overline{\Omega}$. It follows that $\{\cup\mathcal{V}_n : n\in\mathbb{N}\}\in\overline{\Lambda}$ and hence $\cup_{n\in\mathbb{N}}\mathcal{V}_n\in\overline{\Lambda}$. Now the partition $(\mathcal{V}_n)$ witnesses that $\cup_{n\in\mathbb{N}}\mathcal{V}_n\in\overline{\mathcal{O}^{wgp}}$. Thus $\cup_{n\in\mathbb{N}}\mathcal{V}_n\in\overline{\Lambda^{wgp}}$.

$(3)\Rightarrow (4)$. The proof for this implication is modelled in \cite[Theorem 2]{cocVIII}. Here we present complete proof for convenience of the reader. Let $(\mathcal{U}_n)$ be a sequence of $\gamma$-covers of $X$. Without loss of generality we can assume that $\mathcal{U}_m\cap\mathcal{U}_n=\emptyset$ for $m\neq n$. For each $n$ choose $\mathcal{U}_n=\{U_m^{(n)} : m\in\mathbb{N}\}$. Also for each $n$ we define $\mathcal{W}_n=\{U_m^{(1)}\cap\cdots\cap U_m^{(n)} : m\in\mathbb{N}\}\setminus\{\emptyset\}$. Clearly each $\mathcal{W}_n$ is a $\gamma$-cover of $X$. We can further assume that $\mathcal{W}_m\cap\mathcal{W}_n=\emptyset$ for $m\neq n$ by discarding elements, whenever required. We also in advance choose for each element of each $\mathcal{W}_n$ a representation as an intersection as in the definition. Since $X$ satisfies $\Sf(\Gamma,\overline{\Lambda^{wgp}})$, there is a sequence $(\mathcal{H}_n)$ such that for each $n$ $\mathcal{H}_n$ is a finite subset of $\mathcal{W}_n$ and $\cup_{n\in\mathbb{N}}\mathcal{H}_n\in\overline{\Lambda^{wgp}}$. Let $\cup_{n\in\mathbb{N}}\mathcal{H}_n=\cup_{n\in\mathbb{N}}\mathcal{K}_n$, where each $\mathcal{K}_n$ is finite and $\mathcal{K}_m\cap\mathcal{K}_n=\emptyset$ for $m\neq n$, and each finite set $F\subseteq X$ is contained in $\overline{\cup\mathcal{K}_n}$ for some $n$.

We now choose a sequence of positive integers $i_1<i_2<\dotsb$ as follows. Let $i_1\geq 1$ be so small such that  $\mathcal{H}_1\cap\mathcal{K}_j=\emptyset$ for all $j>i_1$.
Now choose $\mathcal{V}_1$ as the set of $U\in\mathcal{U}_1$ that appear as terms (if exist) in the representations of elements of $\mathcal{K}_j, j\leq i_1$.

Next take $i_2>i_1$ so small such that $\mathcal{H}_2\cap\mathcal{K}_j=\emptyset$ for all $j>i_2$. Choose $\mathcal{V}_2$ as the set of $U\in\mathcal{U}_2$ that appear as terms (if exist) in the representations of elements of $\mathcal{K}_j, j\leq i_2$.

Proceeding similarly we obtain a sequence $(\mathcal{V}_n)$ such that for each $n$ $\mathcal{V}_n$ is a finite subset of $\mathcal{U}_n$ and $\mathcal{V}_m\cap\mathcal{V}_n=\emptyset$ for $m\neq n$. Let $F$ be a finite subset of $X$. Then there exists a $n_0\in\mathbb{N}$ such that $F\subseteq\overline{\cup\mathcal{K}_{n_0}}$. Let $k_0\in\mathbb{N}$ be the least such that $n_0\leq i_{k_0}$. It is easy to see that $(\mathcal{H}_1\cup\mathcal{H}_2\cup\cdots\cup\mathcal{H}_{k_0-1})
\cap\mathcal{K}_{n_0}=\emptyset$. This implies that for each $V\in\mathcal{K}_{n_0}$ there exists a $U_V\in\mathcal{U}_{k_0}$ such that $U_V$ is a term in the representation of $V$. If $\mathcal{V}=\{U_V : V\in\mathcal{K}_{n_0}\}$, then $\cup\mathcal{K}_{n_0}\subseteq\cup\mathcal{V}\subseteq\cup\mathcal{V}_{k_0}$. Hence $X$ satisfies $(4)$.

$(5)\Rightarrow (4)$. The proof is similar to the preceding argument with necessary modifications.

$(6)\Rightarrow (5)$. Let $(\mathcal{U}_n)$ be a sequence of $\gamma$-covers of $X$. We can assume that $\mathcal{U}_m\cap\mathcal{U}_n=\emptyset$ for $m\neq n$. Since $X$ satisfies $\Uf(\mathcal{O},\overline{\mathcal{O}^{wgp}})$, there is a sequence $(\mathcal{V}_n)$ such that for each $n$ $\mathcal{V}_n$ is a finite subset $\mathcal{U}_n$ and $\{\cup\mathcal{V}_n : n\in\mathbb{N}\}\in\overline{\mathcal{O}^{wgp}}$. Clearly $\cup_{n\in\mathbb{N}}\mathcal{V}_n\in\overline{\mathcal{O}}$. Now choose $\{\cup\mathcal{V}_n : n\in\mathbb{N}\}=\cup_{n\in\mathbb{N}}\mathcal{H}_n$, where $(\mathcal{H}_n)$ is a sequence of pairwise disjoint finite sets. Using the sequence $(\mathcal{H}_n)$ we can find a sequence $(\mathcal{F}_n)$ of pairwise disjoint finite sets such that for each $n$ $\mathcal{F}_n\subseteq\cup_{n\in\mathbb{N}}\mathcal{V}_n$ with $\cup_{n\in\mathbb{N}}\mathcal{V}_n=\cup_{n\in\mathbb{N}}
\mathcal{F}_n$ and also $\cup\mathcal{F}_n=\cup\mathcal{H}_n$. Let $F$ be a finite subset of $X$. Choose a $n_0\in\mathbb{N}$ such that $F\subseteq\overline{\cup\mathcal{H}_{n_0}}=\overline{\cup
\mathcal{F}_{n_0}}$. It follows that $\cup_{n\in\mathbb{N}}\mathcal{V}_n\in\overline{\mathcal{O}^{wgp}}$ and consequently $X$ satisfies $\Sf(\Gamma,\overline{\mathcal{O}^{wgp}})$.
\end{proof}


Next result is in line (of the implication $(3)\Rightarrow (4)$) of \cite[Theorem 2]{cocVIII} with necessary modifications.
\begin{Th}
\label{T29}
If $X$ is $\aS$, then $X$ satisfies $\Sf(\Gamma,\overline{\Lambda})$ and each large cover of $X$ is a member of $\overline{\mathcal{O}^{wgp}}$.
\end{Th}
\begin{proof}
 Let $(\mathcal{U}_n)$ be a sequence of $\gamma$-covers of $X$. Without loss of generality we can assume that $\mathcal{U}_m\cap\mathcal{U}_n=\emptyset$ for $m\neq n$. Since $X$ is $\aS$, we apply Theorem~\ref{T28-4}. Thus there is a sequence $(\mathcal{V}_n)$ of pairwise disjoint finite sets such that for each $n$ $\mathcal{V}_n\subseteq\mathcal{U}_n$ and for each finite set $F\subseteq X$ there is a $n$ such that $F\subseteq\overline{\cup\mathcal{V}_n}$. Observe that $\{\cup\mathcal{V}_n : n\in\mathbb{N}\}\in\overline{\Omega}$ and hence $\{\cup\mathcal{V}_n : n\in\mathbb{N}\}\in\overline{\Lambda}$. It follows that $\cup_{n\in\mathbb{N}}\mathcal{V}_n\in\overline{\Lambda}$ and consequently $X$ satisfies $\Sf(\Gamma,\overline{\Lambda})$.

For the next part, first we pick a large cover $\mathcal{U}$ of $X$. We then enumerate $\mathcal{U}$ bijectively as $\{U_n : n\in\mathbb{N}\}$. Now  $(\mathcal{W}_n)$ is a sequence of $\gamma$-covers of $X$, where $\mathcal{W}_n=\{\cup_{n<j\leq m} U_j: m\in\mathbb{N}\}$ for each $n$. We assume that $\mathcal{W}_m\cap\mathcal{W}_n=\emptyset$ for $m\neq n$. Again applying Theorem~\ref{T28-4} to $(\mathcal{W}_n)$ to obtain a sequence $(\mathcal{H}_n)$ of pairwise disjoint finite sets such that for each $n$ $\mathcal{H}_n\subseteq\mathcal{W}_n$ and for each finite set $F\subseteq X$ there is a $n$ such that $F\subseteq\overline{\cup\mathcal{H}_n}$.

Now define $k_0=m_0=n_0=1$ and continue as follows.

Choose $m_1=2$. Observe that $\mathcal{H}_1\subseteq\cup_{j\leq m_1}\mathcal{H}_j$. Next choose $n_1\geq m_1$ so small in such a way that if $U_i$ is a term in the representation of an element of $\cup_{j\leq m_1}\mathcal{H}_j$, then $i<n_1$. Again choose $k_1>n_1$ so that if $j\geq k_1$, then the following conditions are satisfied.
\begin{enumerate}[wide=0pt,label={\upshape(\arabic*)},
leftmargin=*]
  \item If $U_i$ is a term in the representation of an element of $\mathcal{H}_j$, then $i\geq n_1$;
  \item $k_1$ is minimal subject to $1$ and $k_1>n_1$.
\end{enumerate}

Next choose $m_2=k_1+1$. Now take $n_2\geq m_2$ so small such that if $U_i$ is a term in the representation of an element of $\cup_{j\leq m_2}\mathcal{H}_j$, then $i<n_2$. Again choose $k_2>n_2$ so that if $j\geq k_2$, then the following conditions are satisfied.
\begin{enumerate}[wide=0pt,label={\upshape(\arabic*)},
leftmargin=*]
  \item If $U_i$ is a term in the representation of an element of $\mathcal{H}_j$, then $i\geq n_2$;
  \item $k_2$ is minimal subject to $1$ and $k_2>n_2$.
\end{enumerate}

For the general case, choose $m_{j+1}=k_j+1$. Next take $n_{j+1}\geq m_{j+1}$ so small such that if $U_l$ is a term in the representation of an element of $\cup_{i\leq m_{j+1}}\mathcal{H}_j$, then $l<n_{j+1}$. Again choose $k_{j+1}>n_{j+1}$ such that if $l\geq k_{j+1}$, then the following conditions are satisfied.
\begin{enumerate}[wide=0pt,label={\upshape(\arabic*)},
leftmargin=*]
  \item If $U_i$ is a term in the representation of an element of $\mathcal{H}_l$, then $i\geq n_{j+1}$;
  \item $k_{j+1}$ is minimal subject to $1$ and $k_{j+1}>n_{j+1}$.
\end{enumerate}

For each $n$ let $\mathcal{K}_n=\cup_{k_{n-1}+1\leq j\leq k_n}\mathcal{H}_j$. It is easy to observe that for each $m$ $\cup\mathcal{K}_n\subseteq\cup_{n_{m-1}\leq i\leq n_{m+1}}U_i$.

By the construction of $\mathcal{H}_i$'s we have for each finite set $F\subseteq X$ either there is a $n$ such that $F\subseteq\overline{\cup\mathcal{K}_{2n-1}}$, or there is a $n$ such that $F\subseteq\overline{\cup\mathcal{K}_{2n}}$. In the first case the partition \[\left(\{U_i : n_{2k-2}\leq i<n_{2k}\}:k\in\mathbb{N}\right)\] witnesses that $\mathcal{U}\in\overline{\mathcal{O}^{wgp}}$.

In the latter case the partition \[\left(\{U_i : n_{2k-1}\leq i<n_{2k+1}\}:k\in\mathbb{N}\right)\] witnesses that $\mathcal{U}\in\overline{\mathcal{O}^{wgp}}$. This completes the proof.
\end{proof}

By Lemma~\ref{L2}, we obtain the following.
\begin{Cor}
\label{C19}
If $X$ is $\aS$, then $X$ satisfies $\Sf(\Omega,\overline{\Lambda})$ and each large cover of $X$ is a member of $\overline{\mathcal{O}^{wgp}}$.
\end{Cor}

For other two cases we obtain the following.
\begin{Th}
\label{T48}
For a space $X$ the following assertions are equivalent.
\begin{enumerate}[wide=0pt,label={\upshape(\arabic*)},leftmargin=*]
  \item $X$ is $\wS$.
  \item $X$ satisfies $\Uf(\Gamma,\Omega_D)$.
  \item $X$ satisfies $\Sf(\Gamma,{\Lambda^{wgp}}_D)$.
  \item For each sequence $(\mathcal{U}_n)$ of $\gamma$-covers of $X$ there is a sequence $(\mathcal{V}_n)$ of pairwise disjoint finite sets such that for each $n$ $\mathcal{V}_n\subseteq\mathcal{U}_n$ and for each finite collection $\mathcal{F}$ of nonempty open subsets of $X$ there is a $n$ such that $U\cap(\cup\mathcal{V}_n)\neq\emptyset$ for all $U\in\mathcal{F}$.
  \item $X$ satisfies $\Sf(\Gamma,{\mathcal{O}^{wgp}}_D)$.
  \item $X$ satisfies $\Uf(\mathcal{O},{\mathcal{O}^{wgp}}_D)$.
\end{enumerate}
\end{Th}

\begin{Th}
\label{T49}
If $X$ is $\wS$, then $X$ satisfies $\Sf(\Gamma,\Lambda_D)$ and each large cover of $X$ is a member of ${\mathcal{O}^{wgp}}_D$.
\end{Th}
\begin{Cor}
\label{C30}
If $X$ is $\wS$, then $X$ satisfies $\Sf(\Omega,\Lambda_D)$ and each large cover of $X$ is a member of ${\mathcal{O}^{wgp}}_D$.
\end{Cor}

\begin{Th}
\label{T50}
For a space $X$ the following assertions are equivalent.
\begin{enumerate}[wide=0pt,label={\upshape(\arabic*)},leftmargin=*]
  \item $X$ is $\wSk$.
  \item $X$ satisfies $\Uf(\Gamma,\Omega^D)$.
  \item $X$ satisfies $\Sf(\Gamma,{\Lambda^{wgp}}^D)$.
  \item For each sequence $(\mathcal{U}_n)$ of $\gamma$-covers of $X$ there is a dense set $Y\subseteq X$ and a sequence $(\mathcal{V}_n)$ of pairwise disjoint finite sets such that for each $n$ $\mathcal{V}_n\subseteq\mathcal{U}_n$ and each finite set $F\subseteq Y$ is contained in $\cup\mathcal{V}_n$ for some $n$.
  \item $X$ satisfies $\Sf(\Gamma,{\mathcal{O}^{wgp}}^D)$.
  \item $X$ satisfies $\Uf(\mathcal{O},{\mathcal{O}^{wgp}}^D)$.
\end{enumerate}
\end{Th}

\begin{Th}
\label{T51}
If $X$ is $\wSk$, then $X$ satisfies $\Sf(\Gamma,\Lambda^D)$ and each large cover of $X$ is a member of ${\mathcal{O}^{wgp}}^D$.
\end{Th}
\begin{Cor}
\label{C31}
If $X$ is $\wSk$, then $X$ satisfies $\Sf(\Omega,\Lambda^D)$ and each large cover of $X$ is a member of ${\mathcal{O}^{wgp}}^D$.
\end{Cor}

\section{Open Problems}
We are unable to answer the following problems during preparation of this paper.
\begin{Prob}
\label{Q5}
Give an example of a $\aS$ and a $\wS$ space which is not $\wSk$.
\end{Prob}

\begin{Prob}
\label{Q1}
If $X$ is a $\wL$ space with cardinality less than $\mathfrak{d}$, then is it a $\wSk$ (or, a $\wS$, $\wM$) space?
\end{Prob}

\begin{Prob}
\label{Q9}
Is every productively $\wL$ space a $\wS$ (or, a $\wM$) space?
\end{Prob}

\begin{Prob}
\label{Q10}
Is every productively $\aL$ space a $\aS$ (or, a $\aM$) space?
\end{Prob}

The following problems in the case of Alster property still remain open.
\begin{Prob}[\!{\cite[Problem 3.8]{MTCP}}]
\label{Q2}
Does there exist an almost Alster space which is not Alster?
\end{Prob}

\begin{Prob}
\label{Q6}
Find conditions under which almost Alster and Alster spaces are equivalent.
\end{Prob}

\begin{Prob}[\!{\cite[Problem 5]{WCPSP}}]
\label{Q8}
Is every productively $\wL$ space a weakly Alster space?
\end{Prob}

\end{document}